\documentclass[10pt]{amsart}

\usepackage{amsmath}
\usepackage{amssymb}
\usepackage{amsfonts}
\usepackage{geometry}
\usepackage{url}

\usepackage{setspace}
\usepackage{amsthm}
\usepackage[all,cmtip]{xy}
\usepackage{amsmath,amscd}
\usepackage{thmtools}
\usepackage{thm-restate}

\newtheorem{theorem}{Theorem}
\newtheorem{lemma}[theorem]{Lemma}
\newtheorem{proposition}[theorem]{Proposition}
\newtheorem{corollary}[theorem]{Corollary}
\newtheorem*{nono-theorem}{Theorem}

\theoremstyle{definition}
\newtheorem{definition}[theorem]{Definition}

\theoremstyle{remark}
\newtheorem{remark}[theorem]{Remark}

\numberwithin{equation}{section}

\newcommand{\Hom}{\mathrm{Hom}}

\newcommand{\Ad}{\mathrm{Ad}}
\newcommand{\ad}{\mathrm{ad}}

\newcommand{\h}{\mathfrak{h}}
\newcommand{\g}{\mathfrak{g}}
\newcommand{\CC}{\mathbb{C}}
\newcommand{\ZZ}{\mathbb{Z}}
\newcommand{\NN}{\mathbb{N}}

\begin{document}

\title{Degeneration of trigonometric dynamical difference equations for quantum loop algebras to trigonometric Casimir equations for Yangians}

\author{Martina Balagovi\' c}
\address{\newline
work completed at: Department of Mathematics\\  University of York\\ UK \newline
and Department of Mathematics \\ University of Zagreb \\ Croatia\newline
currently at: School of Mathematics and Statistics\\ Newcastle University \\ UK}
\thanks{Supported by the EPSRC grant EP/I014071/1}
\email{martina.balagovic@newcastle.ac.uk}

\begin{abstract}
We show that, under Drinfeld's degeneration \cite{D} of  quantum loop algebras to Yangians, the trigonometric  dynamical difference equations \cite{EV} for the quantum affine algebra degenerate to the trigonometric Casimir differential equations \cite{TL1} for Yangians.
\end{abstract}

\maketitle

\section{ Introduction} 
\subsection{The main result}
For $\g$ a simple Lie algebra over $\CC$, the \emph{quantum loop algebra} $U_qL\g$ and the \emph{Yangian} $Y_\hbar \g$ are, respectively, deformations of the loop algebra $UL\g=U\g[z,z^{-1}]$ and the current algebra $U\g[u]$. The algebras $U\g[z,z^{-1}]$ and $U\g[u]$ are related by a change of variables $z=e^u$. More precisely, one can say that the completion of $U\g[z,z^{-1}]$ at $z=1$ is isomorphic to the completion of $U\g[u]$ at $u=0$, or one can consider the weaker degeneration statement that  $U\g[z,z^{-1}]$ has a filtration by the ideal $\left< z-1\right>$ such that the associated graded algebra is $U\g[u]$. Drinfeld stated the corresponding fact for their deformations  $U_qL\g$ and the $Y_\hbar \g$  in \cite{D}; this was elaborated and proved in detail by Guay and Ma in \cite{GM} (who show that the quantum loop algebra degenerates to the corresponding Yangian), and by Gautam and Toledano Laredo in \cite{GTL1} (who prove a stronger claim that the appropriate completions of the quantum loop algebra and Yangian are isomorphic). 

One can attach a quantum version of the Knizhnik-Zamolodchikov equations to the quantum loop algebra and to the Yangian, called trigonometric qKZ and qKZ equations, respectively. Both sets of equations are written using universal R-matrices, and they correspond to each other under the above degeneration. In this paper we study equations commuting with qKZ equations, namely the trigonometric dynamical difference equations for quantum loop algebras from \cite{EV} and the trigonometric Casimir equations for Yangians from \cite{TL1}. We show that they correspond to each other under the degeneration as well. 

The trigonometric dynamical difference equations were first described by Etingof and Varchenko in \cite{EV}. For functions with values in a tensor product of $N$ representations of $U_qL\g$, depending on complex variables $z_1,\ldots z_N$ and a dynamical parameter $\lambda\in \h^*\cong \h$, they are difference equations expressing the shift of $\lambda$ by a fundamental coweight $\lambda \mapsto \lambda+\kappa \omega^{\vee}$ in terms of the dynamical Weyl group operators corresponding to the lattice element $\omega^{\vee}$ in the affine Weyl group. They commute with the trigonometric qKZ equations, which are q-difference equations in parameters $z_i$.

The trigonometric Casimir equations were first described by Toledano Laredo in \cite{TL1}. For a function with values in a tensor product of $N$ representations of $Y_\hbar \g$, depending on $N$ complex variables and a parameter $e^{\lambda}$ in the complex torus $H$ of $\h$, they are differential equations describing the derivative in direction $\omega^{\vee}$ of the function in the $e^{\lambda}$ variable in terms of universal elements of the Yangian, with logarithmic singularities along the root hypertori. They commute with the rational qKZ equations, which are difference equations in parameters $z_i$.

The main result of the paper is the following (see Theorem \ref{main} and Corollaries \ref{MyOp} and \ref{ValOp}).

\begin{nono-theorem}
Under Drinfeld's degeneration of  quantum loop algebras to Yangians, the trigonometric dynamical difference equations for  the quantum affine algebra, describing the shift of variables $f(\lambda) \mapsto f(\lambda+ \kappa \mu)$, degenerate to the trigonometric Casimir differential equation for Yangians, describing the derivative  in direction $d_{\mu}g(\lambda)$. 
\end{nono-theorem}

The following reason for interest in the trigonometric dynamical difference equations, trigonometric Casimir equations, and their relations was explained to me by Pavel Etingof. The paper \cite{MO} explains that the trigonometric Casimir connection can be identified with the quantum connection for the quantum equivariant cohomology of Nakajima quiver varieties. Quantum K-theory for these varieties is not developed yet, but the experts in the field expect that, when passing from cohomology to K-theory for Nakajima quiver varieties, the trigonometric Casimir connection from \cite{MO} will be  replaced by the trigonometric dynamical difference equations. Results in this paper give supporting evidence for this.

\subsection{The context}
This theorem and the equations it concerns fit into a larger framework. Consider the following diagram of algebras associated to a simple Lie algebra $\g$: 

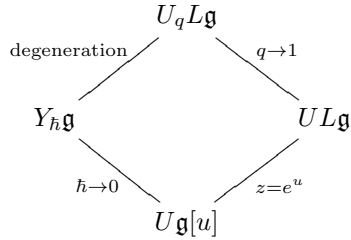
$$\xymatrix{
& U_{q}L \mathfrak{g}  \ar@{-}[dr]^{q\to 1}  \ar@{-}[dl]_{\textrm{degeneration}}   & \\
Y_{\hbar}\mathfrak{g} \ar@{-}[dr]_{\hbar \to 0} &  & UL\mathfrak{g} \ar@{-}[dl]^{z=e^u}\\
& U\mathfrak{g}[u] & 
}
$$

The ``down right" arrows are specializations at $\hbar=0$ and $q=e^{\hbar/2}=1$ . The ``down left" arrows are degenerations. This means that the algebras $UL\mathfrak{g}=U\mathfrak{g}[z,z^{-1}]$ and $U_{q}L \mathfrak{g}$ have filtrations with respect to certain ideals (the kernel of $z\to 1$ for $UL\mathfrak{g}$  and the kernel of $q\to 1,z\to 1$ for  $U_{q}L \mathfrak{g}$), so that the associated graded algebras are $U\mathfrak{g}[u] $ and $Y_{\hbar}\mathfrak{g} $. Additionally, $UL\mathfrak{g}$ and $U_{q}L \mathfrak{g}$ can be completed with respect to these filtrations, and these completions are isomorphic to the completions of graded algebras $U\mathfrak{g}[u] $ and $Y_{\hbar}\mathfrak{g}$  (see \cite{GTL1}, or Section \ref{degeneratesection}).

To each of these algebras, one can associate a system of Knizhnik--Zamolodchikov equations:
$$
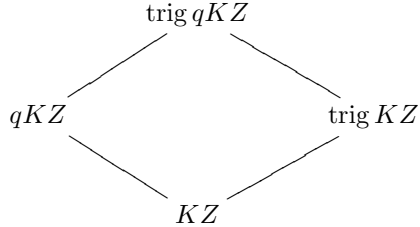
\xymatrix{
&  \mathrm{trig}\, qKZ  \ar@{-}[dr]  \ar@{-}[dl]   & \\
qKZ \ar@{-}[dr] &  &  \mathrm{trig}\, KZ \ar@{-}[dl]\\
& KZ & 
}
$$

These are equations for functions depending on  $N$ distinct complex variables and one dynamical variable in $\h$ or $H$, and taking values in a tensor product of $N$ representations.
 They are all defined using appropriate versions of universal R--matrices. Under the above degenerations and specializations, these universal R--matrices map/degenerate to each other, and consequently so do these systems of equations. (See \cite{FR} and \cite{EFK}).

Additionally, to each of the algebras one can associate a system of equations commuting with the appropriate version of KZ: 

$$
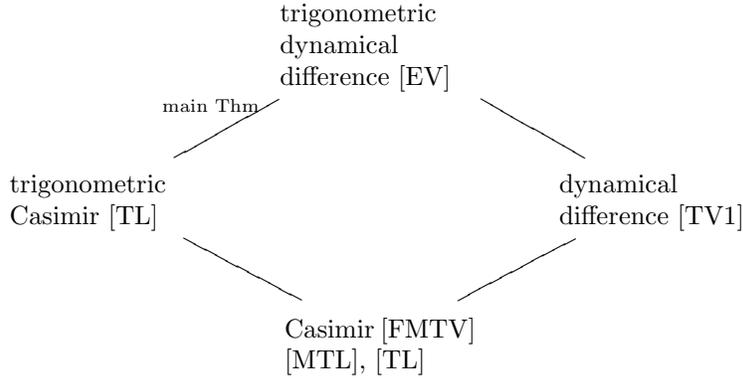
\xymatrix{
& \textrm{\parbox{1in}{\vspace{4pt} trigonometric dynamical difference \cite{EV} } }  \ar@{-}[dr]  \ar@{-}[dl]_{\textrm{main Thm}}  & \\
\textrm{\parbox{1in}{\vspace{4pt}trigonometric Casimir \cite{TL1}}} \ar@{-}[dr] &  &  \textrm{\parbox{1in}{\vspace{4pt} dynamical difference \cite{TV1}}} \ar@{-}[dl]\\
& \textrm{\parbox{1in}{\vspace{4pt}Casimir \cite{FMTV} \cite{MTL}, \cite{TL1} }} & 
}
$$

Trigonometric dynamical difference equations for $U_{q}L \mathfrak{g}$ are described in \cite{EV}, and their definition is a q-generalization of the definition of dynamical difference equations for $UL \mathfrak{g}$ from \cite{TV1}. Trigonometric Casimir equations were defined in \cite{TL1}, where it is shown that they commute with qKZ equations, and that they are a generalization of the corresponding ($\g=\mathfrak{gl}_n$) differential equations compatible with qKZ from \cite{TV2, TV3, MTV}. Different versions of Casimir equations commuting with rational KZ have been independently discovered in \cite{FMTV}, \cite{MTL}, \cite{TL1}, and by De Concini (unpublished). A computation similar to the one presented in this paper describes the degeneration of 
dynamical difference equations to Casimir equations, as was stated in the introduction of \cite{TV1}. This paper completes this diagram of degenerations, and can be seen as a quantized version of the degeneration from \cite{TV1}.

\subsection{The roadmap of the paper}
In Section 2 we fix notation related to the underlying finite dimensional simple Lie algebra $\g$. Section 3 contains all the prerequisites about the quantum affine algebra and the trigonometric  dynamical difference equations. Analogously, Section 4 is about the Yangian, trigonometric Casimir equation and the degeneration of quantum affine algebra to Yangian. The new results of the paper are in the main Section 5. The main result is Theorem \ref{main}, describing how the operators from the quantum affine algebra which define the trigonometric dynamical difference equations degenerate to the operators defining trigonometric Casimir equations when one passes from the quantum affine algebra of $\g$
to the associated Yangian. The interest in this result stems from Corollaries \ref{MyOp}
 and \ref{ValOp}, which relate the trigonometric dynamical difference equations and the trigonometric Casimir equations. Section 6 contains the proof of Theorem \ref{main}.

\subsection*{Acknowledgements}
 I am very grateful to Pavel Etingof for suggesting this problem, explaining the context described in the introduction, and for his guidance throughout the project. I had several very helpful conversations with Valerio Toledano Laredo and Sachin Gautam, who carefully explained both their published work and their work in progress, and shared early drafts of their papers. I also wish to thank Sachin Gautam and the anonymous referee for reading the paper carefully and offering helpful suggestions.

\section{Preliminaries}\label{preliminaries}
Let $\mathfrak{g}$ be a simple Lie algebra  of rank $n$ over $\mathbb{C}$, $\mathfrak{h}$ its Cartan subalgebra, $[a_{ij}]_{i,j}$ the Cartan matrix of finite type encoding this data, and $W$ the corresponding finite Weyl group. Fix a choice $\mathbf{\Delta}=\{ \alpha_1,\ldots \alpha_n \}$ of simple roots in $\mathfrak{h}^*$, and a basis $\{ h_1,\ldots h_n \}$ of $\mathfrak{h}$ such that $\alpha_i(h_j)=a_{ji}$. The Lie algebra $\mathfrak{g}$ has a standard  presentation with Chevalley generators $e_i, f_i, h_i$, $i=1,\ldots n$, and relations
$$[h_i,h_j]=0 \,\,\,\, [h_i,e_j]=a_{ij}e_j \,\,\,\, [h_i,f_j]=-a_{ij}f_j $$
$$[e_i,f_j]=\delta_{ij}h_i, \,\,\,\, \mathrm{ad} (e_i)^{1-a_{ij}}e_j=0, \,\,\,\, \mathrm{ad} (f_i)^{1-a_{ij}}f_j=0$$

Let $d_i$ be the symmetrizing numbers for the Cartan matrix; they are the minimal positive integers which satisfy $d_i a_{ij}=d_j a_{ji}$. There is a unique invariant bilinear form $\left< \cdot, \cdot \right>$ on $\mathfrak{g}$ such that $\left< h_{i} , h_{j}  \right>=d_i^{-1}a_{ji}$, $\left< e_{i} , f_{j}  \right>=\delta_{ij}d_i^{-1}$, and all other pairings of Chevalley basis vectors  are $0$. It induces an identification $\mathfrak{g}\cong \mathfrak{g}^*$. Under this identification, $h_i\mapsto d_i^{-1}\alpha_i$, and $\left< \alpha_{i} , \alpha_{j}  \right>=d_ia_{ij}$. The integers $d_i$ associated to $W$-conjugate simple roots are the same, so defining $d_{\alpha}=\frac{\left< \alpha,\alpha \right>}{2}$ extends this to a function on all roots. 

For any root $\alpha$ let $s_{\alpha}$ be the reflection across the hyperplane $\left< \alpha,\cdot \right>=0$ in $\h^*$. Then the Weyl group $W$ is a finite group generated by simple reflections $s_i=s_{\alpha_i}$, acting on $\mathfrak{h}^*$ by $s_i(\lambda)=\lambda-\lambda(h_i)\alpha_i$ and on $\h$ by $s_i(h)=h-\alpha_i(h)h_i$. Let $BW$ denote the corresponding braid group. Let $\mathbf{R}=W\mathbf{\Delta}$ be the set of roots, and $\mathbf{R}_+$ the subset of positive roots. Let $P=\{\lambda \in \mathfrak{h}^* | \lambda(h_i)\in \mathbb{Z} \}$ be the weight lattice, $Q=\mathrm{span}_{\mathbb{Z}}\mathbf{R}\subseteq P$ the root lattice, and $Q^{\vee}\subseteq P^{\vee}\subseteq \h^*$ their dual lattices.  The coroot lattice $Q^{\vee}$ is spanned by coroots $\alpha_i^{\vee}=\frac{2\alpha_i}{\left< \alpha_{i} , \alpha_{i}  \right>}$, and the coweight lattice $P^{\vee}$ by the fundamental coweights $\omega_i^{\vee}$ such that $\left< \omega_i^{\vee} , \alpha_{j}  \right>=\delta_{ij}$. Let $H=\Hom_{\mathbb{Z}}(P,\mathbb{C}^\times)$ and $T=\Hom_{\mathbb{Z}}(Q,\mathbb{C}^\times)$ be the corresponding complex algebraic tori. 

For any positive root $\alpha$, let $e_\alpha,f_{\alpha}, h_{\alpha}$ be the corresponding $\mathfrak{sl}_2$ triple, with $h_{\alpha}$ corresponding to $\alpha^{\vee}$ under the identification $\g\cong \g^*$. If $\alpha=w(\alpha_i)$, then $h_{\alpha}=w(h_i)$, and $e_\alpha,f_{\alpha}$ are vectors spanning the root spaces $\alpha, -\alpha$, mutually normalized so that $\left< e_{\alpha} , f_{\alpha}  \right>=d_{\alpha}^{-1}=d_{i}^{-1}$. We shall also use another common choice of basis: let $t_{\alpha}=d_{\alpha}h_{\alpha}\in \h$ correspond to $\alpha\in \h^*$ under $\g\cong \g^*$, and let $x_{\alpha}^+$, $x_{\alpha}^{-}$ be a choice of elements proportional to $e_\alpha,f_{\alpha}$ and satisfying $\left< x_{\alpha}^+ , x_{\alpha}^-  \right>=1$. Then $\h$ is spanned by $t_i$, and $U\g$ is generated by $t_i$ and $x_{i}^\pm=x_{\alpha_i}^\pm$.

Let $\theta$ be the highest root of $\g$, $d_0=d_{\theta}=\frac{\left< \theta,\theta \right>}{2}$, and $\theta^{\vee}=d_0^{-1}\theta$ the corresponding coroot. Let $\rho\in \h^*$ be the half sum of positive roots, so that $\rho(h_i)=1$, and let $\rho^{\vee}\in \h$ be such that $\alpha_i(\rho^{\vee})=1$.

Through the paper, let $q$ and $\hbar$ be formal variables, related by $q^2=e^{\hbar}$. Let $q_i=q^{d_i}$. We use the standard notation for quantum integers $[m]_q=\frac{q^m-q^{-m}}{q-q^{-1}}$ and ${a \brack b}_q=\frac{[a]_q!}{[b]_q![a-b]_q!}$.

\section{Quantum loop algebras and the trigonometric dynamical difference equations} 
\subsection{Quantum loop algebra, the loop presentation}
The quantum loop algebra $U_q L\g$ is a deformation of $UL\g=U\g [u,u^{-1}]$. Its loop presentation  is as follows: as an associative algebra over $\CC[[\hbar]]$, it is topologically generated by $E_{i,k},F_{i,k}, H_{i,k}$, for $i=1,\ldots n$, $k\in \ZZ$, with relations: 
$$[H_{i,k}, H_{j,l}]=0$$
$$[H_{i,0},E_{j,k}]=a_{ij}E_{j,k} \,\,\,\,\,\,\,\, [H_{i,0},F_{j,k}]=-a_{ij}F_{j,k} $$
$$[H_{i,r},E_{j,k}]=\frac{[ra_{ij}]_{q_i}}{r}E_{j,k+r} \,\,\,\, \,\,\,\, [H_{i,r},F_{j,k}]=-\frac{[ra_{ij}]_{q_i}}{r}F_{j,k+r} \,\,\,\, \,\,\,\, \textrm{for } r\ne 0$$
$$E_{i,k+1}E_{j,l}-q_i^{a_{ij}}E_{j,l}E_{i,k+1}=q_i^{a_{ij}} E_{i,k}E_{j,l+1}-E_{j,l+1}E_{i,k}$$
$$F_{i,k+1}F_{j,l}-q_i^{-a_{ij}}F_{j,l}F_{i,k+1}=q_i^{-a_{ij}} F_{i,k}F_{j,l+1}-F_{j,l+1}F_{i,k}$$
$$[E_{i,k},F_{j,l}]=\delta_{ij}\frac{\varPsi_{i,k+l}-\varPhi_{i,k+l}}{q_i-q_i^{-1}}$$
Here $\varPsi_{i,k}$ and $\varPhi_{i,k}$ are determined by: 
$$\varPsi_i (v)=\sum_{r\ge 0} \varPsi_{i,r}v^{-r}=q_i^{H_{i,0}}\exp\left( (q_i-q_{i}^{-1})\sum_{s\ge 1} H_{i,s} v^{-s}\right) $$
$$\varPhi_i (v)=\sum_{r\ge 0} \varPsi_{i,-r}v^{r}=q_i^{- H_{i,0}}\exp\left( -(q_i-q_{i}^{-1})\sum_{s\ge 1} H_{i,-s} v^{s}\right) $$
For any $i\ne j$ and any $l, k_1,\ldots k_{1-a_{ij}}$, and the sum over all permutations $\sigma$ of $\{1,2,\ldots 1-a_{ij} \}$, 
$$\sum_{\sigma}\sum_{s=0}^{1-a_{ij}}(-1)^s {1-a_{ij} \brack s}_q E_{i,k_{\sigma(1)}}\ldots E_{i,k_{\sigma(s)}}E_{j,l} E_{i,k_{\sigma(s+1)}}\ldots E_{i,k_{\sigma(1-a_{ij})}}=0 $$ 
$$\sum_{\sigma}\sum_{s=0}^{1-a_{ij}}(-1)^s {1-a_{ij} \brack s}_q F_{i,k_{\sigma(1)}}\ldots F_{i,k_{\sigma(s)}}F_{j,l} F_{i,k_{\sigma(s+1)}}\ldots F_{i,k_{\sigma(1-a_{ij})}}=0 $$ 

At $q=1$, this algebra is $U\g[z,z^{-1}]$ with $E_{i,l}\equiv E_i z^l$, $F_{i,l}\equiv F_i z^l$, $H_{i,l}\equiv H_i z^l$. 

There is an extension $U\hat{\g}$ of $UL\g$ by a central element $c$, and its deformation $U_q\hat{\g}$ (see \cite{CP}). Additionally, to make the root spaces finite dimensional, one can add a derivation $d$, and consider the semidirect product $U_q\tilde{\g}=U_q\hat{\g}\rtimes \CC d$, with $d$ acting on $H_{i,k}, E_{i,k}, F_{i,k}$ by $k$. In that case, $\tilde{\h}=\h \oplus \CC c \oplus \CC d$ takes the role of Cartan. We write the elements of the dual Cartan $\tilde{\h}^*$ as triples $(\lambda,k,\delta)$, with $(\lambda,k,\delta)(h+ac+bd)=\lambda(h)+ak+ b\delta$. The invariant form extends by $\left< d,c \right>=1$, and $d,c$ are orthogonal to all the other generators. The roots of this algebra, called the affine roots, are $\{ (\alpha,0,n) | \alpha \in \mathbf{R} \textrm{ or } \alpha=0, n\in \mathbb{Z}, (\alpha,n)\ne (0,0) \}$. The positive roots are $\{ (\alpha,0,n) | n>0, \alpha \in \mathbf{R}\cup\{ 0 \} \}\cup \{ (\alpha,0,0) | \alpha \in \mathbf{R}_+ \}$.

We are interested in finite dimensional and loop representations, on which $c$ always acts by zero. That is why we do not include it in the definition of the algebra. However, we keep working with the full dual Cartan $\tilde{\h}^*$.


\subsection{Quantum loop algebra, the Kac-Moody presentation}
The algebra $U_{q}L\g$ has another presentation, of the Kac-Moody type. For $\theta$  the highest root of $\g$, consider the extended Cartan matrix of size $n+1$, labeled by $0,\ldots n$, with $a_{00}=2$, $a_{i0}=\frac{-2\left<\alpha_i, \theta \right>}{\left< \alpha_i, \alpha_i \right>}$, $a_{0i}=\frac{-2\left<\alpha_i, \theta \right>}{\left< \theta, \theta \right>}$. 
An alternative presentation of the quantum loop algebra is as an associative algebra over $\CC[[\hbar]]$ with generators $\mathcal{E}_i, \mathcal{F}_i, \mathcal{H}_i$, $i=0,\ldots n$ and relations:
$$[\mathcal{H}_i, \mathcal{H}_j]=0$$
$$[\mathcal{H}_i, \mathcal{E}_j]=a_{ij}\mathcal{E}_j \,\,\,\,\,\,\,\, [\mathcal{H}_i, \mathcal{F}_j]=-a_{ij}\mathcal{F}_j $$
$$[\mathcal{E}_i, \mathcal{F}_j]=\delta_{ij}\frac{q_i^{\mathcal{H}_i}-q_i^{\mathcal{-H}_i}}{q_i-q_i^{-1}}$$
$$\sum_{k=0}^{1-{a_{ij}}} (-1)^k{1-a_{ij} \brack k}_q \mathcal{E}_i^{1-a_{ij}-k} \mathcal{E}_j \mathcal{E}_i^k=0 \,\,\,\,\,\,\,\,  \sum_{k=0}^{1-{a_{ij}}} (-1)^k{1-a_{ij} \brack k}_q \mathcal{F}_i^{1-a_{ij}-k} \mathcal{F}_j \mathcal{F}_i^k=0 $$

The two presentations give isomorphic algebras, but the isomorphism is not unique. Following \cite{B}, we fix a choice of isomorphism in Section \ref{LatticeElts}. The simple root corresponding to $\mathcal{E}_0$ is $\alpha_{0}=(-\theta,0,1) \in \tilde{\h}^*$.



\subsection{Loop representations}
For $\overline{V}$ a finite dimensional representation of $U_qL\g$, its loop representation is $V=\overline{V}[z,z^{-1}]$, with the action $E_{i,k}|_{V}. x=z^k E_{i,k}|_{\overline{V}} . x $, $F_{i,k}|_{V}. x=z^k F_{i,k}|_{\overline{V}} . x $, $H_{i,k}|_{V}. x=z^k H_{i,k}|_{\overline{V}} . x $. Define also the extended loop representation as $V=\overline{V}[z^{1/D},z^{-1/D}]$, where $D$ is the common denominator of $\frac{\left< \omega_i^{\vee},\omega_i^{\vee} \right>}{d_i}$, for $i=1\ldots n$.

The qKZ equations and the trigonometric dynamical difference equations concern functions taking values in (extended) loop representations and in (completed) tensor products of loop representations. We will use the facts that every loop representation has a weight decomposition and that one can define the action of the dynamical Weyl group on them, but we are going to work with the universal operators (i.e. algebra elements)  defining these equations and not their evaluations in loop representations,  so we will not need to work with any particular loop representation directly. 

\subsection{Affine and extended affine Weyl group}

We defined $W$ to be the finite Weyl group associated to $\g$ and generated by  the simple reflections $s_1,\ldots s_n$. The Weyl group associated to $L\g$ and the extended Cartan matrix is the affine Weyl group $W^{aff}$, generated by the simple reflections $s_0,\ldots s_n$. As $s_i$ is a simple reflection with respect to the root $\alpha_i$, its action on the elements of $\tilde{\h}^*$ is:
\begin{align}
i\ne 0,\, \tilde{\alpha_i}=(\alpha_i,0,0): \quad & s_i(\lambda,k,\delta)=(\lambda-\left< \lambda,\alpha_i^\vee\right> \alpha_i, k,\delta) \label{reflect} \\
i= 0,\, \tilde{\alpha_0}=(-\theta,0,1): \quad & s_0(\lambda,k,\delta)=(\lambda-(\left< \lambda,\theta^\vee\right> -k)\theta,k,\delta+\left< \lambda,\theta^\vee\right> -k) \label{reflect0}
\end{align}

There is a well know isomorphism $W^{aff}\cong W\ltimes Q^{\vee}$. For $\nu \in Q^{\vee}$, write  $t^{\nu}=(1,\nu)\in W\ltimes Q^{\vee}$, and identify $w\in W$ with $(w,0)\in W\ltimes Q^{\vee}$. Then $t^{\nu}t^{\mu}=t^{\nu+\mu}$ and $wt^{\nu}w^{-1}=t^{w(\nu)}$. The isomorphism $W^{aff}\cong W\ltimes Q^{\vee}$ is given by $s_i\mapsto s_i$ for $i=1,\ldots n$ and $s_0\mapsto t^{\theta^{\vee}}s_{\theta}$. The action of $t^{\nu}$ on $\tilde{\h}^*$ is:
\begin{align}
t^\nu(\lambda,k,\delta)&=(\lambda+d_0k\nu, k, \delta-\left< \lambda,\nu\right>-d_0k\frac{\left< \nu,\nu \right>}{2}).\label{LatticeAction}
\end{align}

The extended affine Weyl group is $W^{ext}=W\ltimes P^{\vee}$, with the above action extending to $\nu\in P^{\vee}$ by the same formula. This is not a Coxeter group, but it has a well defined notion of length. Let $\Pi$ be the group of elements of length zero. There is a split exact sequence $$1\to W\ltimes Q^{\vee} \to W\ltimes P^{\vee}\to \Pi\to 1,$$  so $\Pi\cong W^{ext}/W^{aff}\cong P^{\vee}/Q^{\vee}$, and $W^{ext}\cong \Pi\ltimes W^{aff}$. The elements of $\Pi$ are parametrized by minuscule fundamental coweights $\omega_i^{\vee}$ (i.e. elements of $P^{\vee}$ such that $\left<\omega_i^{\vee},\alpha_j \right>=\delta_{ij}$ and $\left<\omega_i^{\vee},\theta \right>=1$). Let $w_0$ be the longest element of $W$, $w_{0}^i$ the longest element of the Weyl group generated by $s_1,\ldots s_{i-1}, s_{i+1},\ldots s_n$, and $w_{[i]}=w_0w_0^i$. Then all nontrivial elements of $\Pi$ are $\pi_i=t^{\omega_i^{\vee}}(w_{[i]})^{-1}$. 

Conjugation action of $\Pi$ on $W^{aff}$ permutes simple reflections. In this way, $\Pi$ embeds into the automorphism group of the extended Dynkin diagram of $\g$. For $\pi \in \Pi$ such that $\pi s_i \pi^{-1}=s_j$, write $j=\pi(i)$.

\subsection{Affine and extended affine braid group}\label{braidgpsection}
The affine Weyl group $W^{aff}$ is a Coxeter group, so one can consider the associated braid group $BW^{aff}$. Define the extended affine braid group as $BW^{ext}=\Pi\ltimes BW^{aff}$, using the same $\Pi$ permutation action on the generators $s_i$ as above. Let us continue using notation $s_i$, $t^\mu$ and $\pi$ for elements of the braid group. 

The group $BW^{ext}$ has two presentations, analogous to the two presentations of $W^{ext}$ as $\Pi\ltimes W^{aff}$ and as $W\ltimes P^{\vee}$, see \cite{Macdonald}:
\begin{itemize}
\item $BW^{ext}=\Pi\ltimes BW^{aff}$
\begin{itemize}
\item Generators: $$
\pi \in \Pi,\,\,\, 
s_i, s_i^{-1},  i=0,\ldots n $$
\item Relations:  $$\Pi \textrm{ is a subgroup}$$ 
$$s_{i} \textrm { satisfy braid relations} $$ $$\pi s_i \pi^{-1}=s_{\pi(i)}$$
\end{itemize}
\item $BW^{ext}=B(W\ltimes P^{\vee})$
\begin{itemize}
\item Generators: $$t^{\omega_i^{\vee}}, \,\,\, s_i,s_i^{-1}, i=1\ldots n $$
\item Relations: $$t^{\omega_i^{\vee}}t^{\omega_j^{\vee}}=t^{\omega_j^{\vee}}t^{\omega_i^{\vee}}$$ $$s_{i} \textrm { satisfy braid relations} $$ $$s_i^{-1}t^{\omega_i^{\vee}}s_i^{-1}=t^{s_i(\omega_i^{\vee})}$$ $$s_jt^{\omega_i^{\vee}}=t^{\omega_i^{\vee}} s_j,\,\,\, j\ne i$$
\end{itemize}
\end{itemize}

\subsection{Quantum Weyl group}\label{QuantumWeylgroup}

For any $i=0,\ldots n$, let $(U_{q}\mathfrak{sl}_2)_i$ be the subalgebra of $U_{q}L\g$ generated by $\mathcal{E}_i, \mathcal{F}_i, \mathcal{H}_i$, isomorphic to $U_{q_i}\mathfrak{sl}_2$. The quantum Weyl group element of this algebra is the following element of its completion (see \cite{S}):
\begin{eqnarray}\mathbb{S}_i=\mathbb{S}(\mathcal{H}_i, \mathcal{E}_i, \mathcal{F}_i)&=&\exp_{q_i^{-1}}(-q_i^{-1}\mathcal{F}_iq_{i}^{\mathcal{H}_i}) \exp_{q_i^{-1}}(\mathcal{E}_i)\exp_{q_i^{-1}}(-q_i\mathcal{F}_iq_{i}^{-\mathcal{H}_i}) \, q_i^{\mathcal{H}_i(\mathcal{H}_i+1)/2} \label{quantumWeyl} \\
\nonumber &=& \exp_{q_i^{-1}}(q_i^{-1}\mathcal{E}_iq_{i}^{-\mathcal{H}_i}) \exp_{q_i^{-1}}(-\mathcal{F}_i)\exp_{q_i^{-1}}(q_i\mathcal{E}_iq_{i}^{\mathcal{H}_i}) \, q_i^{\mathcal{H}_i(\mathcal{H}_i+1)/2} 
\end{eqnarray}
where
$\exp_{p}(x)=\sum_{m\ge 0} \frac{1}{[m]_p!} p^{m(m-1)/2}\, x^m.$ An integrable module for $U_{q}L\g$ is a module $V$ which is locally finite for $(U_{q}\mathfrak{sl}_2)_i$ for every $i$. The operators $\mathbb{S}_i$ acts on such modules, mapping the weight space $V[\nu]$ to $V[s_i\nu]$.

The quantum Weyl group elements are invertible, and satisfy the braid relation (see \cite{L, S}), so the map $s_{i}\to \mathbb{S}_i$ can be extended to an action of the braid group on the representation. For  $w=s_{i_1}^{\delta_1}\ldots s_{i_k}^{\delta_k}$, $\delta_i=\pm 1$ a reduced expression in the braid group, define $\mathbb{S}_w=\mathbb{S}_{i_1}^{\delta_1}\ldots \mathbb{S}_{i_k}^{\delta_k}$. For $U_{q}\mathfrak{sl}_2$ a finite dimensional representation $V_m$ with highest weight $m$ and a basis given by $v_{m}\in V_m[m]$, $v_{m-2j}=\frac{f^j}{[j]_q!}v_m$, this action can be computed as $\mathbb{S}v_{m-2j}=(-1)^{m-j}q^{(m-j)(j+1)}v_{2j-m}$. 

Let $T_i$ be the automorphism of $U_{q}L\g$ defined as $T_i(x)=\mathbb{S}_ix\mathbb{S}_i^{-1}$. We shall need its inverse $T_i^{-1}$, which given on generators by the following formulas (see \cite{B,S,CP}):
$$T_i^{-1}(\mathcal{H}_j)=\mathcal{H}_j-a_{ji}\mathcal{H}_i=s_i(\mathcal{H}_j) \,\,\,\,\, \forall j$$
$$T_{i}^{-1}(\mathcal{E}_i)=-q_i^{-\mathcal{H}_i}\mathcal{F}_i   \,\,\,\,\,\,\,\, T_{i}^{-1}(\mathcal{F}_i)=-\mathcal{E}_iq_i^{\mathcal{H}_i} $$
$$T_{i}^{-1}(\mathcal{E}_j)= \sum_{k=0}^{-a_{ij}} (-1)^{k+a_{ij}} q_i^{-k} \frac{\mathcal{E}_i ^{k}}{[k]_{q_i}!}\, \mathcal{E}_j \, \frac{\mathcal{E}_i^{-a_{ij}-k}}{[-a_{ij}-k]_{q_i}!} \,\,\,\,\,\, j\ne i$$
$$T_{i}^{-1}(\mathcal{F}_j)= \sum_{k=0}^{-a_{ij}} (-1)^{k+a_{ij}} q_i^{k} \frac{\mathcal{F}_i ^{-a_{ij}-k}}{[-a_{ij}-k]_{q_i}!}\, \mathcal{F}_j \, \frac{\mathcal{F}_i^{k}}{[k]_{q_i}!} \,\,\,\,\,\, j\ne i.$$

One can extend this to the action of $BW^{ext}$ by defining $${T_\pi}(\mathcal{E}_j)=\mathcal{E}_{\pi(j)} \qquad {T_\pi}(\mathcal{F}_j)=\mathcal{F}_{\pi(j)}  \qquad  {T_\pi}(\mathcal{H}_j)=\mathcal{H}_{\pi(j)}.$$ 

For $w\in BW^{ext}$, $w=\pi s_{i_1}^{\delta_1}\ldots s_{i_l}^{\delta_l}$, let $T_w=T_\pi T_{s_{i_1}}^{\delta_1}\ldots T_{s_{i_l}}^{\delta_l}$. For brevity, write $T_{\omega_i^{\vee}}=T_{t^{\omega_i^{\vee}}}$.




\subsection{Root vectors}
Quantum Weyl group from the previous section can be used to define analogues of the root vectors $e_{\alpha},f_{\alpha}$ for the quantum affine algebra. If $w=s_{i_1}\ldots s_{i_k}$ is a reduced decomposition $\delta$ of $w\in W^{aff}$ such that $s_{i_1}\ldots s_{i_k}s_j$ is also reduced, then $\alpha=w(\alpha_j)$ is a positive root. Define 
$$\mathcal{H}_{\alpha}=T_{i_1}^{-1}\ldots T_{i_k}^{-1}(\mathcal{H}_j)=s_{i_1}\ldots s_{i_k}(\mathcal{H}_j) $$
$$\mathcal{E}_{\alpha}^{\delta}=T_{i_1}^{-1}\ldots T_{i_k}^{-1}(\mathcal{E}_j) \qquad
\mathcal{F}_{\alpha}^{\delta}=T_{i_1}^{-1}\ldots T_{i_k}^{-1}(\mathcal{F}_j).
$$

They satisfy: 
\begin{align*}
[\mathcal{H}_i,\mathcal{E}_{\alpha}^{\delta}]=\alpha({H}_i) \mathcal{E}_{\alpha}^{\delta} \quad & \quad
[\mathcal{H}_i,\mathcal{F}_{\alpha}^{\delta}]=-\alpha({H}_i) \mathcal{F}_{\alpha}^{\delta}.
\end{align*}
The Cartan element $\mathcal{H}_{\alpha}$ does not depend on $\delta$, while $\mathcal{E}_{\alpha}^{\delta}, \mathcal{F}_{\alpha}^{\delta}$ generally do. At $q=1$, the root spaces are one dimensional, so for $\alpha=(\alpha',0,n)\in \tilde{\h}^*$, $\mathcal{E}_{\alpha}^{\delta}$ is a multiple of $e_{\alpha'}z^n$ if $\alpha'>0$ or $f_{\alpha'}z^n$ if $\alpha'<0$. Analogously, $\mathcal{F}_{\alpha}^{\delta}$ is a multiple of $f_{\alpha'}z^{-n}$ or $e_{\alpha'}z^{-n}$. Moreover, as $T_i^{-1}$ at $q=1$ preserves the invariant form $\left< \cdot , \cdot \right>$, we also have $\left< \mathcal{E}_{\alpha}^{\delta} , \mathcal{F}_{\alpha}^{\delta} \right>=\left<e_{\alpha'}z^n, f_{\alpha'}z^{-n} \right>=d_{\alpha}=d_{j}$. So, although  $\mathcal{E}_{\alpha}^{\delta}, \mathcal{F}_{\alpha}^{\delta}$ depend on $\delta$, at $q=1$ their product $\mathcal{E}_{\alpha}^{\delta} \mathcal{F}_{\alpha}^{\delta}$ does not, and is equal to $e_{\alpha'}f_{\alpha'}$ or $f_{\alpha'}e_{\alpha'}$. Analogous elements with the same properties could have been defined using $T_i$ instead of $T_i^{-1}$; in general they are not equal.

\subsection{Lattice elements of the quantum Weyl group}\label{LatticeElts}
Following \cite{B}, let us now use the above braid group action to fix an isomorphism between two presentations of $U_qL\g$. First, choose a sign map $o$, assigning $+1$ or $-1$ to each vertex in the Dynkin diagram of $\g$, in such a way that two adjacent vertices have opposite signs (as the Dynkin diagram of a simple $\g$ is connected, there are exactly two such choices). Rescale $T_{\omega_i^{\vee}}$, the braid group action of $t^{\omega_i^{\vee}}\in BW^{ext}$ by quantum Weyl group operators, by $o(i)$ to get $\hat{T}_{\omega_i^{\vee}}=o(i)T_{\omega_i^{\vee}}.$ Then $$E_{i,k}\mapsto (\hat{T}_{\omega_i^{\vee}})^{-k}(\mathcal{E}_i)$$ $$F_{i,k}\mapsto (\hat{T}_{\omega_i^{\vee}})^{k}(\mathcal{F}_i)$$ for $i=1,\ldots n$, $k\in \mathbb{Z}$ induces an isomorphism between the two presentations of $U_q\g$.

Let $\alpha_i$ be any simple root of the same length as $\theta$, and $w\in W$ such that $w(\alpha_i)=\theta$. Under the inverse of this isomorphism, $$\mathcal{E}_0\mapsto -o(i)q^{-h_{\theta}}T_w(F_{i,1})=T_wT_{\omega_i^{\vee}}T_i^{-1}(E_{i,0})$$
$$\mathcal{F}_0\mapsto -o(i)T_w(E_{i,-1})q^{h_{\theta}}=T_wT_{\omega_i^{\vee}}T_i^{-1}({F}_{i,0}).$$  This assignment of does not depend on $i$, but only on the choice of the sign map $o(i)$. In either case, $\mathcal{H}_0\mapsto -h_{\theta}.$

\subsection{Dynamical Weyl group}

The following three subsections review the definitions and results from \cite{EV} which we are going to need. 

Dynamical Weyl group of a Kac-Moody algebra (in our case: $U_qL\g$) is a collection of operators $A_{w}(\lambda)$, labeled by the braid group elements $w$ (in our case: the affine braid group $BW^{aff}$) and depending rationally on $\lambda$ (for $q=1$) or $q^{\lambda}$ (for $q\ne 1$). For each integrable representation $V$, $A_{w}(\lambda)$ is an operator on $V$ mapping the the weight space $V[\nu]$ to $V[w\nu]$. 

Let $V$ be such a representation, $v\in V[\nu]$ a weight vector, $\lambda$ a dominant integral weight large enough with respect to $V$, and $M_{\lambda}$ the Verma module with highest weight $\lambda$ and highest weight vector $m_{\lambda}$, and $\rho$ the half sum of positive roots. There is a unique intertwining operator $M_{\lambda}\to M_{\lambda-\nu}\otimes V$ of the form
$$\Phi^v (m_{\lambda})=m_{\lambda-\nu}\otimes v+\textrm{terms with lower weight in the first tensor factor}.$$
The restriction of $\Phi^v$ to $M_{w(\lambda+\rho)-\rho}\subseteq M_{\lambda}$ is an intertwining operator $M_{w(\lambda+\rho)-\rho}\to M_{w(\lambda-\nu+\rho)-\rho}\otimes V$. Then $A_{w}(\lambda)v$ is defined to be the unique vector in $V[w\nu]$ such that 
$$\Phi^v (m_{w(\lambda+\rho)-\rho})=m_{w(\lambda-\nu+\rho)-\rho}\otimes A_{w}(\lambda)v+\textrm{terms with lower weight in the first tensor factor}.$$
This defines $A_{w}(\lambda)v$ for $\lambda$ dominant integral and large enough; it depends on $\lambda$ or $q^{\lambda}$ rationally, and is then extended as a rational function of $\lambda$ or $q^{\lambda}$. In case of loop representations which are not integrable, an analogous definition is made using completed tensor products. 

We will mostly need their so-called unshifted versions, the operators $\mathcal{A}_{w}(\lambda)$ defined as $$\mathcal{A}_{w}(\lambda)=A_{w}(-\lambda-\rho+\frac{1}{2}h),$$
where $h$ takes value $\nu$ on the weight space $V[\nu]$.

The dynamical Weyl group operators satisfy braid-like relations: if $w_1,w_2$ are elements of the Weyl group with $l(w_1w_2)=l(w_1)+l(w_2)$, then 
\begin{align}
\mathcal{A}_{w_1w_2}(\lambda)&=\mathcal{A}_{w_1}(w_2\lambda)\, \mathcal{A}_{w_2}(\lambda) \label{Abraid}
\end{align}
Thus, one can always decompose $\mathcal{A}_{w}$ it into a product of $\mathcal{A}_{s_i}$, for $s_i$ simple reflections, using any decomposition of $w$. By defining 
$\mathcal{A}_{s_i^{-1}}(\lambda)=\mathcal{A}_{s_i}(s_i\lambda)^{-1}$, we can associate an operator to any $w$ in the braid group. 

One can define the \emph{(unshifted) dynamical action} of the braid group on the space of functions of $\lambda$ or $q^{\lambda}$ values in a representation $V$  by \begin{align}(w\star f)(\lambda)=\mathcal{A}_w(w^{-1}\lambda)f(w^{-1}\lambda).\label{dynamicalaction1} \end{align} The fact that the dynamical Weyl group operators satisfy braid-like relations means that $\star$ defines a representation of the braid group (in our case, $BW^{aff}$) on this space of functions.

The operator 
$\mathcal{A}_{w}(\lambda)$ 
can always be decomposed as a product of an operator $A^+_{w}:V[\nu]\to V[w\nu]$ which does not depend on $\lambda$, and an operator $\mathbb{B}_w(\lambda)$ which preserves the weight space $V[\nu]$ and depends rationally on $\lambda$ or $q^\lambda$. 
In fact,
\begin{align}
A^+_{s_i}&=(-1)^{h_i}\mathbb{S}_i \nonumber \\
p_q(m,h,e,f)&=\sum_{k=0}^{\infty} \frac{q^{-k(m+2)}}{[k]_q!}f^ke^k\prod_{j=1}^k\frac{1}{[m-h-j+1]_q} \label{pformula} \\
\mathbb{B}_{s_i}(\lambda)&=p_{q_i}((-\lambda-\rho+\frac{1}{2}h)(\mathcal{H}_i),\mathcal{H}_i,\mathcal{E}_i,\mathcal{F}_i) \nonumber \\
\mathcal{A}_{s_i}(\lambda) &=(-1)^{h_i}\,  \mathbb{S}_i \,  \mathbb{B}_{s_i}(\lambda). \label{Asi}
\end{align}
 
To express $A_w(\lambda)$, for arbitrary $w\in W^{aff}$ in a similar way, let $w=s_{i_1}\ldots s_{i_l}\in W^{aff}$ be any reduced decompositon $\delta$, and define $\alpha^l=\alpha_{i_l}$, $\alpha^j=s_{i_l}\ldots s_{i_{j+1}}(\alpha_{i_j})$, $\mathcal{E}^j=T_{i_l}^{-1}\ldots T_{i_{j+1}}^{-1} (\mathcal{E}_{i_j})=\mathcal{E}^{\delta}_{\alpha^j}$, $\mathcal{F}^j=T_{i_l}^{-1}\ldots T_{i_{j+1}}^{-1} (\mathcal{F}_{i_j})=\mathcal{F}^{\delta}_{\alpha^j}$. Then
\begin{align}
\mathbb{B}_w(\lambda)&=\prod_{j=1}^l p_{q_{i_j}}((-\lambda+\frac{1}{2}h)(\mathcal{H}_{\alpha^j})-1,\mathcal{H}_{\alpha^j}, \mathcal{E}^j, \mathcal{F}^j), \label{Bw}
\end{align}
 where the product is noncommutative and the indices increase from left to right, and
\begin{align}
\mathcal{A}_{w}(\lambda)&=(-1)^{\rho^{\vee}-w\rho^{\vee}} \, \mathbb{S}_w \, \mathbb{B}_w(\lambda). \label{Aw} 
\end{align}

\subsection{Operators $\Pi_i^q$}\label{operatorsPi}
For a fundamental coweight $\omega_i^{\vee} \in P^{\vee}$, write $t^{\omega_i^{\vee}}=\pi_i  w_{i} \in W^{ext}\cong \Pi \ltimes W^{aff}$. For example, if $\omega_i^{\vee}$ is minuscule, $t^{\omega_i^{\vee}}=\pi_i \cdot w_{i}$, $w_{i}=w_{[i]}=w_0w^i_0\in W$, and we obtain all $\pi\ne 1\in \Pi$ in this way. On the other extreme, if $\omega_i^{\vee} \in Q^{\vee}$, then $t^{\omega_i^{\vee}}=1\cdot w_i$, $w_i=t^{\omega_i^\vee}\in W^{aff}$.

This gives us a collection of elements $\pi_1,\ldots \pi_n \in \Pi$ (some of them can be repeated and some of them are trivial $1\in \Pi$), and $w_1,\ldots, w_n\in W^{aff}$. \cite{EV}  define the operators $\Pi_i^q$, which are operator versions of $\pi_1,\ldots \pi_n \in \Pi$. (The indexing $i=1,\ldots ,n$ corresponds to all fundamental coweights, not just the minuscule ones.)

At $q=1$, $\Pi_i$ is the element of (the completed) algebra which acts on any extended loop representation $V[z^{1/D},z^{-1/D}]$ as $$\Pi_i=z^{-\omega_i^{\vee}}(A^+_{w_{i}})^{-1}.$$ Let $\xi_i=\mathrm{Ad}(\Pi_i) $ be the $UL\g$ automorphism given by $\xi_i(a)=\Pi_i a \Pi_i^{-1}$. One can show that, up to a constant, $\xi_i$ equals a permutation of indices given by $\pi_i$. More precisely, there is $x_i\in \tilde{\h}$ such that 
\begin{equation} \label{xi}
\xi_i=T_{\pi_i}\circ Ad(e^{x_i}).
\end{equation}
Label the constants $c_{ij}=e^{\alpha_j(x_i)}$.

At $q\ne 1$, define the automorphism $\xi_i$ by the same formula (\ref{xi}). Then one can show there is a unique $\Pi_i^q$, deforming $\Pi_i$, satisfying  
\begin{equation}\label{conjugbypi}
\xi_i(a)=\Pi^q_i a (\Pi^q_i)^{-1} 
\end{equation}
\begin{equation*}
\Pi^q_i|_{q=1}=\Pi_i \,\,\,\, \textrm{and} \,\,\,\, \det(\Pi^q_i|_{z=1})\textrm{ does not depend on } q.
\end{equation*}

\subsection{Trigonometric dynamical difference equations}\label{ddeq}
Now we can write the trigonometric dynamical difference equations. These are difference equations for functions $\varphi$, depending on $\tilde{\lambda} \in \tilde{\h}^*$ and $N$ complex numbers $\mathbf{z}=(z_1,\ldots z_N)$, and taking values in a completed tensor product of extended loop representations of $U_qL\g$, $V=\overline{V_{1}}\otimes \ldots \otimes \overline{V_{N}}[[\frac{z_2}{z_1}, \ldots \frac{z_N}{z_{N-1}} ]][z_1^{1/D},z_1^{-1/D}, \ldots z_N^{1/D},z_N^{-1/D}]$. There is an interest in those $\varphi$ which satisfy the trigonometric qKZ equations. These are q-difference equations,  written using dynamical R-matrices, and concerning shifts in $z_i\mapsto qz_i$. Trigonometric dynamical difference equations describe shifts in the $\lambda$ variable. 
The reason they are interesting lies in the fact that they commute with qKZ. 

Let $\tilde{\lambda}=(\lambda,k, \delta) \in \tilde{\h}^*$. The operators we are about to define do not depend on $\delta$, and $k$ is fixed. Let $\kappa=d_0 k$. The action of $t^{\omega_i^\vee}\in W^{ext}$ on $\tilde{\h}^*$ is
$$t^{\omega_i^\vee}(\lambda,k, \delta)=(\lambda+\kappa \omega_i^\vee, k, \delta-\left<\lambda,\omega_i^\vee \right>-\kappa \frac{\left<\omega_i^\vee,\omega_i^\vee \right>}{2});$$
in particular it shifts the $\lambda$ component by $\lambda\mapsto \lambda+\kappa \omega_i^\vee$.

Remember that $t^{\omega_i^{\vee}}=\pi_i  w_{i}$, and
define the operators 
\begin{equation}\mathcal{A}_i(\mathbf{z}, \lambda)=\Pi_i^q \mathcal{A}_{w_{i}}(\lambda).\label{A_i}\end{equation} 
These operators satisfy 
\begin{align} \mathcal{A}_i(\mathbf{z}, \lambda+\kappa \omega_j^{\vee})\mathcal{A}_j(\mathbf{z}, \lambda)=\mathcal{A}_j(\mathbf{z}, \lambda+\kappa \omega_i^{\vee})\mathcal{A}_i(\mathbf{z}, \lambda)\label{Aomegaicommute}
\end{align} and commute with qKZ equations (see \cite{EV}).
The trigonometric dynamical difference equations are
\begin{equation}\varphi(\mathbf{z}, \lambda+\kappa \omega_i^{\vee})=\mathcal{A}_i(\mathbf{z}, \lambda)\varphi(\mathbf{z}, \lambda).\label{DDeqns}
\end{equation}

If we define \emph{dynamical action} of $P^{\vee}$ on the space of functions $\varphi$ of $q^{\lambda}$ with values in a representation by 
\begin{align} 
(t^{\omega_i^{\vee}}\star \varphi)(\lambda)&=\mathcal{A}_i(\mathbf{z}, t^{-\omega_i^{\vee}}(\lambda))\varphi(t^{-\omega_i^{\vee}}(\lambda)),\label{dynamicalaction2}
\end{align}
then (\ref{Aomegaicommute}) means that this is indeed a representation of $P^{\vee}$, and the trigonometric dynamical difference equations (\ref{DDeqns}) are the conditions for a function $\varphi$ to be invariant with respect to this action, $\forall \mu \in P^{\vee}$ $t^{\mu}\star \varphi = \varphi$.


\section{Yangians and the trigonometric Casimir equation}

\subsection{Drinfeld's new realization of Yangians}
The Yangian $Y_{\hbar}\g$ is a formal deformation of $U\g[u]$. We shall use Drinfeld's new presentation as its definition, see \cite{CP} or \cite{TL1}.

The Yangian  $Y_\hbar \g$ is the associative unital algebra over $\CC[\hbar]$ with generators $X_{i,r}^{\pm}$, $T_{i,r}$, $i=1\ldots n, r\in \mathbb{N}_0$, and relations:
$$[T_{i,r},T_{j,s}]=0 \,\,\,\,\,\,\,\, [T_{i,0},X_{j,s}^{\pm}]=\pm d_ia_{ij}X_{j,s}^{\pm} \,\,\,\,\,\,\,\,  [X_{i,r}^+,X_{j,s}^-]=\delta_{ij}T_{i,r+s}$$
$$ [T_{i,r+1},X_{j,s}^{\pm}]- [T_{i,r},X_{j,s+1}^{\pm}]=\pm \frac{\hbar}{2} d_i a_{ij} (T_{i,r}X_{j,s}^{\pm}+X_{j,s}^{\pm}T_{i,r} )$$
$$ [X_{i,r+1}^{\pm},X_{j,s}^{\pm}]- [X_{i,r}^{\pm},X_{j,s+1}^{\pm}]=\pm \frac{\hbar}{2} d_i a_{ij} (X_{i,r}^{\pm}X_{j,s}^{\pm}+X_{j,s}^{\pm}X_{i,r}^{\pm})$$
$$\sum_{\sigma} \ad{X_{i,r_{\sigma(1)}}^{\pm}} \ad{X_{i,r_{\sigma(2)}}^{\pm}}\ldots \ad{X_{i,r_{\sigma(1-a_{ij})}}^{\pm}} X_{j,s}^{\pm}=0,$$
where in the last relation $i\ne j$, $r_1,\ldots r_{1-a_{ij}} \in \mathbb{N}_0$ are arbitrary, and the sum is over all $\sigma$ permutations of $\{ 1,2,\ldots 1-a_{ij}\}$.

This is a graded algebra, $Y_\hbar \g=\oplus_{r\ge 0} (Y_\hbar \g)_r$, with the grading given by $\deg X_{i,r}^{\pm}=\deg T_{i,r}=r$, $\deg \hbar=1$. Degree $0$ graded part $(Y_\hbar \g)_0$ is canonically identified with $U\g$, by $t_i\mapsto T_{i,0}$, $x_{i}^{\pm}\mapsto X_{i,0}^{\pm}$, and we will sometimes abuse notation by identifying them. At $\hbar=0$, this algebra becomes $U\g[u]$, with $X_{i,r}^{\pm}\equiv x_{i}^{\pm}u^r, \, T_{i,r}\equiv t_iu^r.$

\subsection{Yangians at $\hbar=1$ and the homomorphism $\Psi$}\label{Psi}
By evaluating $\hbar$ to a particular numerical value $a$, one gets an algebra  $Y_a\g$ over $\CC$. For $a,b\ne 0$, the algebras $Y_a\g$ and $Y_b\g$ are isomorphic, and one frequently considers $Y_1\g$ instead of $Y_\hbar\g$. Let $\Psi: Y_\hbar\g\to Y_1\g\otimes \CC[\hbar]$ be the $\CC[\hbar]$-algebra isomorphism defined on generators as $X_{i,r}^{\pm}\mapsto \hbar^r X_{i,r}^{\pm}$, $T_{i,r}\mapsto \hbar^r T_{i,r}$. It maps $(Y_\hbar\g)_{\ge n}$ to $\hbar^n Y_1\g\otimes \CC[\hbar]$.

\subsection{Trigonometric Casimir equation}
This section describes the trigonometric Casimir equations, as defined in \cite{TL1}.

Remember that $H$ is the maximal torus of the connected simply connected complex algebraic group $G$, corresponding to the Cartan subalgebra $\h\subseteq \g$. Let $H^{\textrm{reg}}=H\setminus \cup_{\alpha \in \mathbf{R}} \{ e^{\alpha}=1\}$. 
Let $$K_{\alpha}=x_\alpha^+x_\alpha^-+x_\alpha^-x_\alpha^+$$
 Let $T(t)_1=\sum_i \omega_i^{\vee} (t_i) T_{i,1}$ be an extension by linearity of the map $\h\hookrightarrow (Y_{\hbar}\g)_1$ given by $t_i\mapsto T_{i,1}$. Define $\tau: \h\to (Y_{\hbar}\g)_1$   
\begin{equation}\tau(t)=
-2\left(T(t)_1 -\frac{\hbar}{2}\sum_i \omega_i^{\vee} (t) t_i^2  \right) \label{tau} \end{equation}
For $\mu\in \h^*$, and $h_{\mu}\in \h$ corresponding to it under the isomorphism $\h^*\cong \h$, define $T(\mu)_1=T(h_{\mu})_1$, and $\tau(\mu)=\tau(h_{\mu})$. 
\begin{theorem}[\cite{TL1}]\label{Casimir1}
Consider the trivial vector bundle with the base $H^{\textrm{reg}}$ and the fibre $Y_{\hbar}\g$. 
The connection on this bundle given by
\begin{eqnarray*}
\nabla &=& d-\frac{1}{b}\left( \sum_{\alpha\in \mathbf{R}_+}\frac{d\alpha}{e^\alpha-1}K_{\alpha}+\sum_i d\alpha_i \tau(\omega_i^{\vee})\right) \\
\end{eqnarray*}
is flat and $W$-equivariant for any $b\in \mathbb{C}^\times$.
 \end{theorem}
 
 \begin{definition}The trigonometric Casimir equations are differential equations for flat sections of a trigonometric Casimir connection, i.e. for functions $\varphi$ satisfying $\nabla \varphi=0.$
\end{definition}

\begin{remark}
\begin{enumerate}\item 
The statement that this connection is $W$-equivariant makes sense because it takes values in the zero weight space of $Y_\hbar\g$ with respect to the $\g$ action given by the adjoint 
action of $\g \to Y_\hbar \g $, and W acts on the zero weight space of any integrable representation.
\item This connection takes values in an infinite dimensional algebra $Y_{\hbar}\g$. However, we will be interested in evaluating it in a finite dimensional representation $V$, and considering the flat,  $W$-equivariant connection this induces on a trivial $H^{reg}$ bundle with fibers $V$.
\item In particular, the $Y_1\g$ valued connection obtained from the connection in Theorem \ref{Casimir1} by applying $\Psi$ is also flat and  $W$-equivariant. The equations for its flat sections are
\begin{align}
d_\mu \varphi (\lambda)
&=\frac{1}{b}\left( \sum_{\alpha\in \mathbf{R}_+}\frac{\left< \alpha, \mu \right>}{e^{\left< \alpha, \lambda \right>}-1}K_{\alpha}-2 T(\mu)_1+\sum_i \left<\omega_i^{\vee},\mu \right> t_i^2 \right)\varphi (\lambda).
\label{Casimir2}
\end{align}
\end{enumerate}
\end{remark}

\subsection{Degeneration of the quantum loop algebra to Yangian}\label{degeneratesection}

The paper \cite{D} states that the quantum loop algebra degenerates to Yangian, while \cite{GTL1} give a stronger statement that their completions are isomorphic. To state that more precisely, let $\widehat{Y_{\hbar}\g}=\prod_{n\in \NN_0} (Y_{\hbar}\g)_n$ be the completion of  $Y_{\hbar}\g=\oplus_{n\in \NN_0} (Y_{\hbar}\g)_n$ with respect to its grading. Let $\mathcal{J}$ be the kernel of the composition 
$$ U_qL\g \xrightarrow{q\to 1}   UL\g=U\g[z,z^{-1}]  \xrightarrow{z\to 1}  U\g.$$
This ideal induces a filtration of $U_qL\g$ by $U_qL\g\supseteq \mathcal{J} \supseteq \mathcal{J}^2\supseteq \ldots$. Let $\widehat{U_qL\g}=\varprojlim U_qL\g/\mathcal{J}^n$ be the completion with respect to that filtration. 
\begin{theorem}[\cite{GTL1}]\label{ValSachiso}
There exists an explicit $\CC[[\hbar]]$-algebra homomorphism $\Phi: U_qL\g\to \widehat{Y_{\hbar}\g}$ such that: 
\begin{enumerate}
\item $\Phi$ induces an isomorphism $\widehat{\Phi}:\widehat{U_qL\g}\xrightarrow{\sim} \widehat{Y_{\hbar}\g}$ ;
\item $\Phi$ induces Drinfeld's degeneration $\mathrm{gr}\Phi: \mathrm{gr} U_qL\g \xrightarrow{\sim}  Y_{\hbar}\g$.
\end{enumerate}
\end{theorem}

Let us now describe $\Phi$ from the theorem. First, define elements $t_{i,r}$ of $Y_{\hbar}\g$ by 
$$\hbar \sum_{r\ge 0}t_{i,r}v^{-r-1}= \log \left( 1+\hbar \sum_{r\ge 0} T_{i,r}v^{-r-1} \right).$$ Let $\sigma_i$ be the shift operators on $Y_{\hbar}\g$ defined on the generators as
 $\sigma_i T_{j,k}=T_{j,k}$, $\sigma_i X_{j,k}^{\pm}=X_{j,k+\delta_{ij}}^{\pm}$. Define $g_{i,m}$ with $i=1\ldots n$ and $m\in \NN_0$ by the power series
 $$\sum_{m\ge 0} g_{i,m}v^m=g_{i}(v)=\left(\frac{\hbar}{q_i-q_i^{-1}} \right) ^{1/2} \exp \left( \frac{\hbar}{2}\sum_{r\ge 0} \frac{t_{i,r}}{r!}(-\partial_v)^{r+1} \left(\log\left(\frac{v}{e^{v/2}-e^{-v/2}} \right) \right)  \right)$$

Then $\Phi$ is defined on the loop generators of $U_qL\g$ as 
\begin{align*}
\Phi(H_{i,0}) &= d_i^{-1}t_{i,0}= h_i \in \g \subset Y_{\hbar}\g\\
\Phi(H_{i,r})&=\frac{\hbar}{q_i-q_i^{-1}}\sum_{k\ge 0} t_{i,k}\frac{r^k}{k!}, \,\,\,\,\,\, r\ne 0\\
\Phi(E_{i,r})&= e^{r\sigma_i} g_i(\sigma_i) X_{i,0}^+=e^{r\sigma_i} \sum_{m\ge 0} g_{i,m}X_{i,m}^+ \\
\Phi(F_{i,r})&= e^{r\sigma_i} g_i(\sigma_i) X_{i,0}^-=e^{r\sigma_i} \sum_{m\ge 0} g_{i,m}X_{i,m}^-.
\end{align*}

It is easy to check that at $\hbar=0$, $q=1$, this map $\Phi: U\g[z,z^{-1}]\to  U\g[[u]]$ is given by $\Phi(xz^n)=xe^{nu}$.

\begin{remark}
As the main theorem of this paper concerns the degeneration of the quantum affine algebra to Yangian and not the full isomorphism between their completions (which is a stronger statement), it would be possible to prove it using the results from \cite{GM}. However, in its proof we use the explicit map $\Phi$ for convenience. 
\end{remark}


\section{The main theorem}

\subsection{Statement}\label{statement}

In Section \ref{ddeq} we wrote the operators $\mathcal{A}_i(\mathbf{z},\lambda) \in \widehat{U_q L\g}$, and the difference equations for a shift by a minuscule coweight, $\varphi(\mathbf{z}, \lambda+\kappa\omega_i^{\vee})=\mathcal{A}_i(\mathbf{z}, \lambda)\varphi(\mathbf{z}, \lambda).$ 
Let $\mu \in P^{\vee}$ be an arbitrary coweight, and write $\mu =\sum_{i=1}^n m_i \omega_i^{\vee}$, with $m_i=\left< \mu, \alpha_i \right> \in \mathbb{Z}$. Define the operator $\mathcal{A}^{\mu}(\mathbf{z},\lambda)$ of shift by $\mu$ as a composition of $\mathcal{A}_i(\mathbf{z},\lambda)$ inductively on $\sum_{i=1}^n \left| m_i \right|$ as follows:
\begin{eqnarray*}
\textrm{if } \mu=\mu'+\omega_i^{\vee} & \,\,\,\,\,\,& \mathcal{A}^{\mu}(\mathbf{z},\lambda)=\mathcal{A}^{\mu'}(\mathbf{z},\lambda+\kappa\omega_i^{\vee})\mathcal{A}_i(\mathbf{z},\lambda)
=\mathcal{A}^{\mu'}(\mathbf{z},t^{\omega_i^{\vee}} (\lambda))\mathcal{A}_i(\mathbf{z},\lambda)\\
\textrm{if } \mu=\mu'-\omega_i^{\vee} & \,\,\,\,\,\,& \mathcal{A}^{\mu}(\mathbf{z},\lambda)=\mathcal{A}^{\mu'}(\mathbf{z},\lambda-\kappa\omega_i^{\vee})\mathcal{A}_i^{-1}(\mathbf{z},\lambda-\kappa\omega_i^{\vee})=\mathcal{A}^{\mu'}(\mathbf{z},t^{-\omega_i^{\vee}} (\lambda))\mathcal{A}_i^{-1}(\mathbf{z},t^{-\omega_i^{\vee}} (\lambda))\\
\end{eqnarray*} 
This is well defined and independent on the ordering of the multiset $\{m_i\omega_i \} _i$, due to (\ref{Aomegaicommute}).
In particular, if $\mu=\mu_+-\mu_-$ with $\mu_+,\mu_-\in \sum_{i}\mathbb{N}_0\omega_i$ both dominant, then 
\begin{align}\label{dominant}
\mathcal{A}^{\mu}(\mathbf{z},\lambda)&=\mathcal{A}^{\mu_+}(\mathbf{z},\lambda-\kappa\mu_-)\mathcal{A}^{\mu_-}(\mathbf{z},\lambda-\kappa\mu_-)^{-1}.
\end{align}

The operator $\mathcal{A}^{\mu}(\mathbf{z},\lambda)$ lies in $\widehat{U_q L\g}$. If $\varphi$ is a function with values in a representation such that $\varphi$ satisfies the system of trigonometric dynamical difference equations $\varphi(\mathbf{z}, \lambda+\kappa \omega_i^{\vee})=\mathcal{A}_i(\mathbf{z}, \lambda)\varphi(\mathbf{z}, \lambda),$  then $\varphi$ also satisfies
 $$ \varphi(\mathbf{z}, \lambda+\kappa\mu)=\mathcal{A}^\mu(\mathbf{z}, \lambda)\varphi(\mathbf{z}, \lambda)$$ for any $\mu \in P^{\vee}$. We shall mostly focus on the case when $\mu \in Q^{\vee} \subseteq P^{\vee}$. The main theorem of this paper is the following. 

\begin{theorem}\label{main}
Let $\mu \in Q^{\vee}$ be any coroot of $\g$, and $\mathcal{A}^{\mu}(\mathbf{z},\lambda) \in \widehat{U_q L\g}$ the operator of shift by $\mu$. The image under $\Phi$ in the filtered algebra $\widehat{Y_\hbar \g}$ of the rescaled operator $\mathcal{A}^{\mu}(\mathbf{z},\frac{\lambda}{\hbar}) $ is  
$$\Phi (\mathcal{A}^{\mu}(\mathbf{z},\frac{\lambda}{\hbar}) )=1+\frac{1}{2}\left( \hbar \sum_{\alpha \in \mathbf{R}_+} \frac{\left<\alpha,\mu\right>}{e^{-\left< \lambda, \alpha \right>}-1} 2x_{\alpha}^-x_{\alpha}^+ +\tau(h_{\mu}) \right) \mod Y_{\hbar}\g_{\ge 2},$$
with $\tau:\h\to (Y_{\hbar}\g)_{1}$ given by (\ref{tau}) as
$$\tau(h)=-2T(h)_1+\hbar\sum_i \omega_i^{\vee}(h)t_i^2.$$
\end{theorem}

\subsection{Degeneration of difference equations to differential equations}

We are mostly interested in the implications Theorem \ref{main} has on the differential and difference equations defined by the operators from Theorem \ref{main}. Let ${V}_i$, $i=1\ldots N$ be finite dimensional representations of $Y_1 \g$, and consider their pullbacks $\Psi^* V_i$ to representations of $Y_\hbar \g$ and  $(\Psi\circ\Phi)^* V_i$ to representations of $U_qL\g$, and associated loop representations. All functions and functional equations in this subsection take values in the (completed) tensor product of these loop representations, considered as a representation of $U_qL\g$ or $Y_1 \g$. Let $$D_\mu(\lambda) =\frac{1}{2\kappa}\left( -2T(h_\mu)_1+\sum_i \omega_i^{\vee}(h_\mu)t_i^2 + \sum_{\alpha \in \mathbf{R}_+} \frac{\left<\alpha,\mu\right>}{e^{-\left< \lambda, \alpha \right>}-1} 2x_{\alpha}^-x_{\alpha}^+ \right)\in Y_1\g.$$

\begin{corollary}\label{MyOp}
Under Drinfeld's degeneration \cite{D} described in Section \ref{degeneratesection}, trigonometric dynamical difference equations from \cite{EV}, described by (\ref{DDeqns})
$$f(\lambda+\kappa \mu)=\mathcal{A}^{\mu}(\mathbf{z},\lambda)f(\lambda)$$ 
degenerate to differential equations
$$d_\mu g(\lambda) =D_\mu (\lambda) g(\lambda),$$
in the sense that $$\lim_{\hbar\to 0} \Psi\circ \Phi \left( \frac{\mathcal{A}^{\mu}(\mathbf{z},\frac{\lambda}{\hbar})-1 }{\hbar \kappa}\right)=D_\mu (\lambda).$$
\end{corollary}
\begin{proof}
 Theorem \ref{main}, the fact that $\Psi(T(h_\mu)_1)=\hbar T(h_\mu)_1$, and  $\Psi(Y_\hbar\g_{\ge 2}) =\hbar^2 Y_1\g \otimes \CC[\hbar] $ directly imply that
 \begin{equation*} (\Psi \circ \Phi) (\mathcal{A}^{\mu}(\mathbf{z},\frac{\lambda}{\hbar}))=1+\frac{\hbar}{2}\left( -2T(h_\mu)_1+\sum_i \omega_i^{\vee}(h_\mu)t_i^2 + \sum_{\alpha \in \mathbf{R}_+} \frac{\left<\alpha,\mu\right>}{e^{-\left< \lambda, \alpha \right>}-1} 2x_{\alpha}^-x_{\alpha}^+ \right) \mod \hbar^2.  
 \end{equation*}
 This can be rewritten as
 $$\frac{(\Psi \circ \Phi) (\mathcal{A}^{\mu}(\mathbf{z},\frac{\lambda}{\hbar}))-1}{\hbar \kappa}=D_\mu (\lambda)  \mod \hbar, $$
 which is the claim of the corollary.
\end{proof}

The equation $d_\mu g(\lambda) =D_\mu(\lambda) g(\lambda)$ 
is equivalent to the trigonometric Casimir equation (\ref{Casimir2}), in the sense that they differ by rescaling and a change of variables $\lambda\mapsto -\lambda$.  

\begin{corollary}\label{ValOp}
If $g(\lambda)$ satisfies the differential equation $d_\mu g(\lambda) =D_\mu(\lambda) g(\lambda),$ from Corollary \ref{MyOp}, then $$\tilde{g}(\lambda) =g(-\lambda) \cdot \prod_{\alpha\in \mathbf{R}_+}\left(e^{-\left< \lambda, \alpha \right>} -1\right)^{-\frac{1}{2\kappa}t_\alpha}$$ satisfies the trigonometric Casimir equation  (\ref{Casimir2})  with $b=-2\kappa$, given by
\begin{align*}
d_\mu \tilde{g}(\lambda) =\frac{1}{-2\kappa}\left( -2T(h_\mu)_1+\sum_i \omega_i^{\vee}(h_\mu)t_i^2 + \sum_{\alpha \in \mathbf{R}_+} \frac{\left<\alpha,\mu\right>}{e^{\left< \lambda, \alpha \right>}-1} K_{\alpha} \right)\tilde{g}(\lambda).
\end{align*}
\end{corollary}
\begin{proof}
This is a direct calculation using that $2x_{\alpha}^-x_{\alpha}^+=x_{\alpha}^-x_{\alpha}^++x_{\alpha}^+x_{\alpha}^- -t_{\alpha}=K_{\alpha} -t_{\alpha}$, and that multiplying by $\prod_{\alpha\in \mathbf{R}_+}\left(e^{-\left< \lambda, \alpha \right>} -1\right)^{\frac{1}{-2\kappa}t_\alpha}$ adds a summand 
$\frac{1}{-2\kappa}\sum_{\alpha \in \mathbf{R}_+} \frac{\left<\alpha,\mu\right>}{e^{\left< \lambda, \alpha \right>}-1} t_{\alpha}$ to the right hand side of the equation. 
\end{proof}

\begin{remark}
Many questions related to solutions of dynamical difference equations and the trigonometric Casimir equations are open. (For example calculating their monodromy is interesting.) The behaviour of these functions at infinity can in general be complicated. Let us state these corollaries with additional (nontrivial!) simplifying assumptions on the asymptotic behaviour of the solutions.

Let $f(\lambda)$ be a function of $q^{\lambda}$, $\lambda \in \h^*$ with values in a (completed) tensor product of loop representations associated to ${V}_i$. 
Assume $f$ satisfies the trigonometric dynamical difference equations, $f(\lambda+\kappa \mu)=\mathcal{A}^{\mu}(\mathbf{z},\lambda)f(\lambda),$ and
that (this is a nontrivial assumption!) $g(\lambda)=\lim_{\hbar\to 0} f(\frac{\lambda}{\hbar})$ exists and is well behaved. For example, it is enough to assume that $f(\frac{\lambda}{\hbar})=g(\lambda)+\hbar g_1(\lambda)+O(\hbar^2)$, with $g,g_1$ continuous. It can be thought of as a function of $e^\lambda\in H$. Pass to functions with values in a representation of $Y_1\g$ on the same tensor product of loop representations. The difference equation for $f(\frac{\lambda}{\hbar})$ is
$$f(\frac{\lambda+\hbar \kappa \mu}{\hbar})=\left( \Psi\circ \Phi \right) \left( \mathcal{A}^{\mu}(\mathbf{z},\frac{\lambda}{\hbar})\right) f(\frac{\lambda}{\hbar})=
\left( 1+\hbar \kappa D_\mu(\lambda) \right) f(\frac{\lambda}{\hbar}) \mod \hbar^2.$$
After rearranging and taking $\lim_{\hbar \to 0}$, this becomes $d_\mu g(\lambda) =D_\mu(\lambda) g(\lambda).$

So, under the appropriate assumptions on the behaviour of $f(\frac{\lambda}{\hbar})$ around infinity, the solution $f(\lambda)$ of the trigonometric dynamical difference equations produces a solution $$\tilde{g}(\lambda) =\lim_{\hbar\to 0}f( -\frac{\lambda}{\hbar}) \cdot \prod_{\alpha\in \mathbf{R}_+}\left(e^{-\left< \lambda, \alpha \right>} -1\right)^{-\frac{1}{2\kappa}t_\alpha}$$ of the  trigonometric Casimir equation  (\ref{Casimir2}) with $b=-2\kappa$.

\end{remark}

\section{Proof of the main theorem}
The rest of the paper is the proof of Theorem \ref{main}.

\subsection{Dynamical Weyl operators and the extended braid group}


We saw in (\ref{dynamicalaction1}) and (\ref{dynamicalaction2}) that by assigning operators $\mathcal{A}_w(\lambda)$ and $\mathcal{A}_i(\mathbf{z},\lambda)$ to $w\in BW^{aff}\subseteq  BW^{ext}$ and $\omega_i^{\vee}\in P^{\vee}\subseteq  BW^{ext}$ one can define the dynamical action $\star$ of $BW^{aff}$ and $P^{\vee}$ on the space of functions of $\lambda$ with values in a representation of $U_qL\g$. We would like to extend this action to the whole $BW^{ext}$, through assignments $\pi_i\mapsto \Pi_i^q$, $s_i\mapsto \mathcal{A}_{s_i}(\lambda)$. However, it turns out that this is not a representation of $BW^{ext}$, as the relations from $BW^{ext}$ hold only up to a constant.

\begin{proposition}\label{BraidRelations}
The operators $\mathcal{A}_{s_i}(\lambda)$ and $\Pi_i^q$ satisfy
\begin{align}
\Pi_i^q \Pi_j^q&=\Pi_j^q \Pi_i^q \label{1.4} \\ 
\mathcal{A}_{s_i}(\lambda) & \textrm{ satisfy braid like relations } \label{1.5} \\
\Pi_j^q  \mathcal{A}_{s_i}(\lambda)  (\Pi_j^q)^{-1}&=
(c_{ji})^{\mathcal{H}_i} \mathcal{A}_{\pi_j(s_i)}(\pi_j(\lambda)) \label{1.6}
\end{align}
where $c_{ji}=e^{\alpha_i(x_j)}$ is given by (\ref{xi}), and in particular $(c_{ji})^{\mathcal{H}_i}$ is a scalar function on the weight spaces not depending on $\lambda$ nor $q$. Comparing this to relations in \ref{braidgpsection} defining $BW^{ext}$, $\pi_i\mapsto \Pi_i^q$, $s_i\mapsto \mathcal{A}_{s_i}(\lambda)$ assigns an operator to every reduced decomposition of $w\in BW^{ext}$, and operators assigned to two reduced decompositions of the same element differ by  a multiplicative scalar. \end{proposition}
\begin{proof}
{\it The proof of (\ref{1.4})} The proof that $\Pi_i^q$ and $\Pi_j^q$ commute is similar to the argument in the second half of Theorem 62 from \cite{EV}. 

As in Section \ref{operatorsPi}, let $\pi_i=t^{\omega_i^{\vee}}w_{i}^{-1}$, $\pi_j=t^{\omega_j^{\vee}}w_{j}^{-1} \in \Pi$ with $w_i,w_j\in W$. As $\Pi\cong P^{\vee}/Q^{\vee}$ is commutative, 
$$ t^{\omega_i^{\vee}+w_{i}^{-1}(\omega_j^{\vee})}w_{i}^{-1}w_{j}^{-1} =t^{\omega_i^{\vee}}w_{i}^{-1}t^{\omega_j^{\vee}}w_{j}^{-1}=\pi_i\pi_j=\pi_j\pi_i=t^{\omega_j^{\vee}+w_{j}^{-1}(\omega_i^{\vee})}w_{j}^{-1}w_{i}^{-1} \in P^{\vee}\ltimes W\cong W^{ext},$$
so $$\omega_i^{\vee}+w_{i}^{-1}(\omega_j^{\vee})=\omega_j^{\vee}+w_{j}^{-1}(\omega_i^{\vee}) \in P^{\vee}$$ $$w_{i}^{-1}w_{j}^{-1} =w_{j}^{-1}w_{i}^{-1} \in W.$$
At $q=1$, $\Pi_i^q=z^{-\omega_i^{\vee}}(A_{w_i}^+)^{-1}$, so 
$$\Pi_i\Pi_j=z^{-\omega_i^{\vee}}(A_{w_i}^+)^{-1}z^{-\omega_j^{\vee}}(A_{w_j}^+)^{-1}=z^{-(\omega_i^{\vee}+w_{i}^{-1}(\omega_j^{\vee}))}(A_{w_jw_i}^+)^{-1}=$$
$$=z^{-(\omega_j^{\vee}+w_{j}^{-1}(\omega_i^{\vee}))}(A_{w_iw_j}^+)^{-1}=\Pi_j\Pi_i.$$
So, $\Pi_i^q$ and $\Pi_j^q$ commute at $q=1$. This implies that the operators $\xi_i=\Ad(\Pi_i)$ from (\ref{xi}) commute at $q=1$. At $q\ne 1$, $\xi_i, \xi_j$ they are given by the same formulas, so they commute as well. This means that $(\Pi_j^q)^{-1}(\Pi_i^q)^{-1}\Pi_j^q\Pi_i^q$ commutes with every generator of the algebra, and in particular acts by some constant $c'_{ij}(q)\in \mathbb{C}$ depending on $q$ on every irreducible representation. Taking determinants of both sides of the equation $(\Pi_j^q)^{-1}(\Pi_i^q)^{-1}\Pi_j^q\Pi_i^q=c'_{ij}(q)$, using that $c'_{ij}(1)=1$, and using that by definition (\ref{conjugbypi}) the determinat of $Pi_i^q|_{z=1}$ doesn't depend on $q$, we get that $c'_{ij}(q)=1$ and (\ref{1.4}) follows.

{\it The proof of (\ref{1.5})} This is (\ref{Abraid}), and was shown in \cite{EV}.

{\it The proof of (\ref{1.6})} Let us show that $\Pi_j^q \mathcal{A}_{s_i}(\lambda) (\Pi_j^q)^{-1}=\xi_j \left( \mathcal{A}_{s_i}(\lambda) \right)=  c_{ji}^{\mathcal{H}_i}\, \mathcal{A}_{\pi_j(s_i)}(\pi_j(\lambda)),$
where $x_j\in \h$ are given by (\ref{xi}). First, write $\mathcal{A}_{s_i}(\lambda)=(-1)^{\mathcal{H}_i} \, \mathbb{S}_i \, \mathbb{B}_{s_i}(\lambda)$ as in (\ref{Asi}). By (\ref{xi}) and (\ref{conjugbypi}), conjugation by $(\Pi_j^q)$ is given by $\xi_j=T_{\pi_j} \circ \mathrm{Ad}(e^{x_j})$. Now compute, using the fact that $\mathbb{S}_i$ maps the weight space $V[\nu]$ to $V[s_i \nu]$ and $\Pi_j^{q}$ maps $V[\nu]$ to $V[\pi_j \nu]$:
\begin{eqnarray*} \xi_j((-1)^{\mathcal{H}_i})&=&(-1)^{\xi_j(\mathcal{H}_i)}=(-1)^{\mathcal{H}_{\pi_j(i)}}\\ 
\xi_j(\mathbb{S}_i)&=&T_{\pi_j}( e^{x_j} \mathbb{S}_{i}e^{-x_j})=(e^{x_j-s_i^{-1}x_j})T_{\pi_j}(\mathbb{S}_{i})=
e^{\alpha_i(x_j)\mathcal{H}_j}\mathbb{S}_{\pi_j(i)}=c_{ji}^{\mathcal{H}_j}\mathbb{S}_{\pi_j(i)}\\
\xi_j(\mathbb{B}_{s_i}(\lambda))&=& \xi_j( p((-\lambda-\rho+\frac{1}{2}h)(\mathcal{H}_i), \mathcal{H}_i, \mathcal{E}_i,\mathcal{F}_i) \\
&=& p((-\lambda-\rho+\frac{1}{2}\pi_j^{-1} (h))(\mathcal{H}_i), \xi_j(\mathcal{H}_i), \xi_j(\mathcal{E}_i),\xi_j(\mathcal{F}_i)) \\
&=&  p((-\pi_j(\lambda)-\rho+\frac{1}{2}h)(\mathcal{H}_{\pi_j(i)} ), \mathcal{H}_{\pi_j(i)}, \mathcal{E}_{\pi_j(i)},\mathcal{F}_{\pi_j(i)}) \\
&=& \mathbb{B}_{\pi_j(s_i)}(\pi_j(\lambda)).
\end{eqnarray*}
Putting those together, we get (\ref{1.6}).  Notice that $c_{ji}^{\mathcal{H}_j}$ is a scalar, taking value $c_{ji}^{\mathcal{\nu(H}_j)}$ on $V[\nu]$, and not depending on $\lambda, q$ or $\mathbf{z}$. 
\end{proof}



\subsection{Decomposing the operator $\mathcal{A}^\mu$ associated to a dominant $\mu \in Q^{\vee}$}
Let $\mu \in Q^{\vee}$, and assume $\mu$ is in the dominant chamber. We shall first prove Theorem \ref{main} for dominant coroots.
 We shall present $\mu$ in three different ways, and use the interplay between presentations (\ref{presentation1}), (\ref{presentation2}) and (\ref{presentation3}) to calculate the degeneration of $\mathcal{A}^{\mu}(\mathbf{z},\lambda)$.

Firstly, as it is in particular a coweight in $P^{\vee}$, it can be written as
\begin{align}\label{presentation1}
\mu=\sum_{i}m_i\omega_i^{\vee} \in P^{\vee}, & \,\,\,\,\, m_i=\left<\mu, \alpha_i \right>.
\end{align}
Dominant means that $m_i\ge 0$. In section \ref{statement} we used this way of expressing $\mu$ to define the operator $\mathcal{A}^{\mu}(\mathbf{z},\lambda)$ of shift by $\mu$ in terms of the operators $\mathcal{A}^{\omega_i^{\vee}}(\mathbf{z}, \lambda)=\mathcal{A}_i(\mathbf{z}, \lambda)$.
 
Secondly,  $\mu \in Q^{\vee}=\sum_i \mathbb{Z}\alpha_i^{\vee}$ can be written as 
\begin{align}\label{presentation2}
\mu=\sum_{i}k_i\alpha_i^{\vee} \in Q^{\vee}, & \,\,\,\,\, k_i=d_i \left<\mu, \omega_i^{\vee} \right>.
\end{align}

Thirdly, $\mu$ can be seen as an element of $BW^{aff}$, and thus expressed as a product of $s_i, s_i^{-1}$ for $i=1\ldots n$. As it is dominant, it can be written using only $s_i$. Let us thus choose one such reduced expression and write
\begin{align}\label{presentation3}
\mu=s_{i_1}\ldots s_{i_l}\in BW^{aff}\subseteq \Pi \ltimes BW^{aff}\cong BW^{ext} \end{align}
The corresponding dynamical Weyl group operator is $\mathcal{A}_{\mu}(\lambda)=\mathcal{A}_{s_{i_1}\ldots s_{i_l}}(\lambda)$.

As a corollary of Proposition \ref{BraidRelations} we get
\begin{corollary}
$$\mathcal{A}^{\mu}(\mathbf{z},\lambda)=C'\cdot \mathcal{A}_{t^\mu}(\lambda)=C\cdot \mathbb{S}_{i_1}\ldots \mathbb{S}_{i_l} \cdot \mathbb{B}_{s_{i_1}\ldots s_{i_l}}(\lambda),$$
where $C',C$ are scalar functions, acting as constants 
on every weight space, and not depending on  $\lambda, q$ or $\mathbf{z}$. \label{WriteProduct}
\end{corollary}
\begin{proof}
The first equality follows from Proposition \ref{BraidRelations}.
To get the second one, use formula (\ref{Aw}), which states that $\mathcal{A}_{t^\mu}(\lambda)=(-1)^{\rho^{\vee}-s_{i_1}\ldots s_{i_l}\rho^{\vee}}\mathbb{S}_{i_1}\ldots \mathbb{S}_{i_l}\mathbb{B}_{s_{i_1}\ldots s_{i_l}}(\lambda)$. 
\end{proof}

Note that, in writing this decomposition, we used the fact that $s_{i_1}\ldots s_{i_l}\in BW^{aff}$ is expressed using only $s_i$ and not $s_i^{-1}$; these are the only elements of $BW^{aff}$ for which the decomposition of $\mathcal{A}_{s_i}(\lambda)=(-1)^{h_i}\mathbb{S}_i\mathbb{B}_{s_i}(\lambda)$ was defined. This is the reason for assuming $\mu$ is dominant. 

The aim is to calculate $\Phi(\mathcal{A}^{\mu}(\mathbf{z},\frac{\lambda}{\hbar})) \mod Y_{\hbar}\g_{\ge 2}$. We shall now do so for each of the factors of Corollary \ref{WriteProduct}.

\subsection{The degeneration of $\mathbb{B}_{s_{i_1}\ldots s_{i_l}}(\frac{\lambda}{\hbar})$ }\label{B}
In this subsection we shall calculate $\Phi(\mathbb{B}_{s_{i_1}\ldots s_{i_l}}(\frac{\lambda}{\hbar})) \mod Y_{\hbar}\g_{\ge 2}$.

Remember that $t^{\mu}=s_{i_1}\ldots s_{i_l}\in BW^{aff}$ is a reduced decomposition, and consider $l$ affine roots defined by $\tilde{\alpha}^{j}=s_{i_l}\ldots s_{i_{j+1}}(\alpha_{i_j})$, for $j=1,\ldots l$. The action of $s_i$ on the affine roots is given by (\ref{reflect}) and (\ref{reflect0}). Let $\tilde{\alpha}^{j}=(\alpha^{j},0,n^{j})$, and let $\mathcal{E}^j=T_{i_l}^{-1}\ldots T_{i_{j+1}}^{-1}(\mathcal{E}_{i_j})$, $\mathcal{F}^j=T_{i_l}^{-1}\ldots T_{i_{j+1}}^{-1}(\mathcal{F}_{i_j})$, $\mathcal{H}^j=T_{i_l}^{-1}\ldots T_{i_{j+1}}^{-1}(\mathcal{H}_{i_j})=s_{i_l}\ldots s_{i_{j+1}}(\mathcal{H}_{i_j})$, as in formula (\ref{Bw}) and the comments preceding it. Using (\ref{Bw}) and (\ref{pformula}), the operator $\mathbb{B}_{s_{i_1}\ldots s_{i_l}}(\frac{\lambda}{\hbar})$, when acting on $V[\nu]$, can be written as
\begin{align*}
\mathbb{B}_{s_{i_1}\ldots s_{i_l}}(\frac{\lambda}{\hbar})|_{V[\nu]}&=\prod_{j=1}^l p_{q_{i_{j}}}((-\frac{\lambda}{\hbar}+\frac{1}{2}\nu)(\mathcal{H}^j)-1,\mathcal{H}^j, \mathcal{E}^j, \mathcal{F}^j) \\
&=\prod_{j=1}^l \left( \sum_{k=0}^{\infty} \frac{1}{[k]_{q_{i_j}}!}(\mathcal{F}^j)^k (\mathcal{E}^j)^k\prod_{i=1}^k\frac{q_{i_j}^{(\frac{1}{\hbar} \lambda(\mathcal{H}^j)-\frac{1}{2}\nu(\mathcal{H}^j)-1)}}{[-\frac{1}{\hbar} \lambda(\mathcal{H}^j)-\frac{1}{2}\nu(\mathcal{H}^j)-i]_{q_{i_j}}} \right).
\end{align*}
We shall now calculate $\Phi$ of this, up to elements of $Y_{\hbar}\g$ of degree $\ge2$.

\begin{lemma}
$$\frac{q_{i_j}^{(\frac{1}{\hbar} \lambda(\mathcal{H}^j)-\frac{1}{2}\nu(\mathcal{H}^j)-1)}}{[-\frac{1}{\hbar} \lambda(\mathcal{H}^j)-\frac{1}{2}\nu(\mathcal{H}^j)-i]_{q_{i_j}}} =\hbar d_{i_j}  \frac{1}{e^{-\left< \lambda, \alpha^j \right>}-1} \,\,\,\,  \mod \hbar ^2.$$
\end{lemma}
\begin{proof}
Using that $q_{i_j}=q^{d_{i_j}}=e^{d_{i_j} \hbar/2}=1+\frac{\hbar}{2}d_{i_j} \mod \hbar^2$,
\begin{align*}
\frac{q_{i_j}^{(\frac{1}{\hbar} \lambda(\mathcal{H}^j)-\frac{1}{2}\nu(\mathcal{H}^j)-1)}}{[-\frac{1}{\hbar} \lambda(\mathcal{H}^j)-\frac{1}{2}\nu(\mathcal{H}^j)-i]_{q_{i_j}}} &=\frac{\exp(\frac{d_{i_j}}{2}\lambda(\mathcal{H}^j))\,\,  q_{i_j}^{-\frac{1}{2}\nu(\mathcal{H}^j)-1}\,\, (q_{i_j}-q_{i_j}^{-1})}{\exp(-\frac{d_{i_j}}{2}\lambda(\mathcal{H}^j))\,\, q_{i_j}^{-\frac{1}{2}\nu(\mathcal{H}^j)-i}-\exp(\frac{d_{i_j}}{2}\lambda(\mathcal{H}^j))\,\, q_{i_j}^{\frac{1}{2}\nu(\mathcal{H}^j)+i}} \\
&=\frac{ q_{i_j}^{-\frac{1}{2}\nu(\mathcal{H}^j)-1}\,\, (q_{i_j}-q_{i_j}^{-1})}{\exp(-\lambda(d_{i_j} \mathcal{H}^j))\,\, q_{i_j}^{-\frac{1}{2}\nu(\mathcal{H}^j)-i}- q_{i_j}^{\frac{1}{2}\nu(\mathcal{H}^j)+i}} \\
&=\frac{\left(  1-\frac{\hbar}{2}d_{i_j}(\frac{1}{2}\nu(\mathcal{H}^j)+1) \right)\cdot \left( \hbar d_{i_j}\right) }{\exp(-\left< \lambda, \alpha^j\right>) \cdot \left(1-\frac{\hbar}{2}d_{i_j}(\frac{1}{2}\nu(\mathcal{H}^j)+i)\right)- \left(1+\frac{\hbar}{2}d_{i_j}(\frac{1}{2}\nu(\mathcal{H}^j)+i)\right)} \mod \hbar^2 \\
&=\frac{\hbar d_{i_j}}{e^{-\left< \lambda, \alpha^j\right>}-1} \mod \hbar^2 . \\
\end{align*}
\end{proof}

\begin{corollary}\label{13}
$$\Phi(p_{q_{i_j}}((-\frac{\lambda}{\hbar}+\frac{1}{2}\nu)(\mathcal{H}^j)-1,\mathcal{H}^j, \mathcal{E}^j, \mathcal{F}^j))=1+\hbar \Phi(\mathcal{F}^j \mathcal{E}^j)\cdot \frac{d_{i_j}}{e^{-\left< \lambda, \alpha^j\right>}-1} \mod Y_{\hbar}\g_{\ge 2}.$$
\end{corollary}


\begin{lemma}\label{14}
For any $j=0,\ldots l$, $$\Phi(\mathcal{F}^j \mathcal{E}^j)=\left\{
 \begin{array}{ll} f_{\alpha^j}e_{\alpha^j},& \alpha_j> 0 \\ e_{-\alpha^j}f_{-\alpha^j},& \alpha_j< 0 \\  \end{array} \right. \mod Y_{\hbar}\g_{\ge 1}.$$
\end{lemma}
\begin{proof}
First assume $i\in \{1,\ldots n \}$, so that $\mathcal{E}_i=E_{i,0}$ and $\mathcal{F}_i=F_{i,0}$ are generators corresponding to the affine root $(\alpha_i,0,0)$ under the isomorphism of two presentations of the quantum affine algebras fixed in \ref{LatticeElts}. Looking at the definition of $\Phi$ after theorem \ref{ValSachiso}, we see that $g_i(v)=d_{i}^{-1/2} \mod \hbar$. Thus, 
\begin{align*}
\Phi(\mathcal{F}_{i}\mathcal{E}_{i}) = \Phi(F_{i,0} E_{i,0}) =(d_{i}^{-1/2} X_{i,0}^-) (d_{i}^{-1/2} X_{i,0}^+)=d_{i}^{-1}x_{i}^-x_{i}^+=f_ie_i 
\mod Y_{\hbar}\g_{\ge 1}.
\end{align*}

For $i=0$ and the affine root $(-\theta, 0, 1)$, by \ref{LatticeElts} we have $\mathcal{E}_0=-o(i) f_{\theta}z \mod \hbar$ and $\mathcal{F}_0=-o(i) f_{\theta}z^{-1} \mod \hbar$. Thus, 
\begin{align*}
\Phi(\mathcal{F}_0 \mathcal{E}_0)&= \Phi((e_{\theta}z)(f_{\theta}z^{-1})) \quad \mod \hbar \\
&= e_{\theta} f_{\theta}  \quad \mod Y_{\hbar}\g_{\ge 1}.
\end{align*}

So, the claim holds if $\tilde{\alpha}^j$ is a simple affine root. In general, $\tilde{\alpha}^j=s_{i_l}\ldots s_{i_{j+1}}(\alpha_{i_j})$, $\mathcal{E}^j=T_{i_l}^{-1}\ldots T_{i_{j+1}}^{-1}(\mathcal{E}_{i_j})$ and $\mathcal{F}^j=T_{i_l}^{-1}\ldots T_{i_{j+1}}^{-1}(\mathcal{F}_{i_j})$, with $i_j\in \{0,1,\ldots n \}$. As $T_i^{-1}$ maps any root space $\g[\alpha]$ to $\g[s_{i}(\alpha)]$, and $\Phi(\mathcal{F}_{i_j} \mathcal{E}_{i_j})\in \g[-\alpha_{i_j}]\g[\alpha_{i_j}] \subseteq Y_{\hbar}\g$, we can conclude that $$\Phi(\mathcal{F}^j\mathcal{E}^j)=T_{i_l}^{-1}\ldots T_{i_{j+1}}^{-1}(\Phi(\mathcal{F}_{i_j} \mathcal{E}_{i_j})) \in \g[-\alpha^j] \g[\alpha^j] \subseteq Y_{\hbar}\g.$$
The space $\g[-\alpha^j] \g[\alpha^j] $ is one dimensional, and spanned by either $ f_{\alpha^j}e_{\alpha^j}$ (if $\alpha_j> 0$) or $e_{-\alpha^j}f_{-\alpha^j}$, if $\alpha_j< 0$. The operators $T_i^{-1}$ also preserve the invariant inner product, under which $\left<f_{\alpha^j},e_{\alpha^j} \right>=\left<f_{\alpha_{i_j}},e_{\alpha_{i_j}} \right>=d_{\alpha^j}^{-1}$. This fixes the normalization and implies the claim. 
\end{proof}

\begin{lemma}\label{15}
As multisets, 
$$\{\alpha^j \, | \,  j=1,\ldots l \} = \{ \left<\alpha, \mu \right>  \alpha \,  | \,  \alpha \in \mathbf{R}_{+} \}. $$
\end{lemma}
\begin{proof}
As $t^{\mu}=s_{i_1}\ldots s_{i_l}$ is a reduced presentation, 
\begin{align*}
\{\tilde{\alpha}^j \, | \,  j=1,\ldots l \} &= \{\tilde{\alpha}=(\alpha,0,n) \textrm{ affine root }  | \tilde{\alpha}>0, t^{\mu}(\tilde{\alpha})<0 \} \qquad \textrm{(see \cite{H})}  \\
&= \{(\alpha,0,n)\textrm{ affine root } | (\alpha,0,n)>0, \, (\alpha,0,n-\left<\alpha, \mu \right>) <0 \} \qquad \textrm{(see (\ref{LatticeAction}))}.  \\
\end{align*}

Positive affine roots are $(\alpha,0,n)$ for which $\alpha$ is a root or $0$ and $n\in \mathbb{Z}$ satisfies $n>0$, and the roots $(\alpha,0,0)$ for which $\alpha$ is a positive root. As $\mu$ is dominant, $\left<\alpha, \mu \right>\ge 0$ for any positive root $\alpha$. For $(\alpha,0,n)$ a positive root, there are three cases:
\begin{itemize}
\item If $\alpha<0$, then $n>0$, and $n-\left<\alpha, \mu \right>\ge n>0$, so $t^{\mu}(\tilde{\alpha})>0$.
\item If $\alpha=0$, then $t^{\mu}(\tilde{\alpha})=\tilde{\alpha}>0$.
\item If $\alpha>0$, then $n-\left<\alpha, \mu \right> < 0$ whenever $n<\left<\alpha, \mu \right>$. 
\end{itemize}
So, the set of positive affine roots $\tilde{\alpha}=(\alpha,0,n)$ such that $t^{\mu}(\tilde{\alpha})$ is a negative affine root is
$$\{(\alpha,0,n)| \alpha \in \mathbf{R}_{+}, n=0, \ldots , \left<\alpha, \mu \right>-1  \},$$
and the corresponding multiset of non-affine roots is
$$\{  \left<\alpha, \mu \right>   \alpha | \alpha \in \mathbf{R}_{+} \}.$$
\end{proof}

Putting it all together, we get
\begin{lemma}\label{16}
$$\Phi(\mathbb{B}_w(\frac{\lambda}{\hbar}))=1+\hbar \sum_{\alpha \in \mathbf{R}_+} \frac{\left<\alpha,\mu\right>}{e^{-\left< \lambda, \alpha \right>}-1} x_{\alpha}^-x_{\alpha}^+  \mod Y_{\hbar}\g_{\ge 2}. $$

\end{lemma}
\begin{proof}
\begin{align*}
\Phi(\mathbb{B}_w(\frac{\lambda}{\hbar}))&=\Phi(\prod_{j=1}^l p_{q_{j}}((-\frac{\lambda}{\hbar}+\frac{1}{2}\nu)(\mathcal{H}^j)-1,\mathcal{H}^j, \mathcal{E}^j, \mathcal{F}^j)) \\
&=\prod_{j=1}^l \left( 1+\hbar \Phi(\mathcal{F}^j \mathcal{E}^j)\cdot \frac{d_{i_j}}{e^{-\left< \lambda, \alpha^j\right>}-1}\right) \mod Y_{\hbar}\g_{\ge 2} \quad \textrm{(using Corollary \ref{13})} \\
&=1+\hbar \sum_{j=1}^l \frac{d_{i_j}}{e^{-\left< \lambda, \alpha^j\right>}-1} f_{\alpha^j}e_{\alpha^j}  \mod Y_{\hbar}\g_{\ge 2} \quad \textrm{(using Lemma \ref{14}}) \\
&=1+\hbar \sum_{\alpha \in \mathbf{R}_+} \frac{\left<\alpha,\mu\right>}{e^{-\left< \lambda, \alpha \right>}-1} x_{\alpha}^-x_{\alpha}^+  \mod Y_{\hbar}\g_{\ge 2} \quad \textrm{(using Lemma \ref{15})}.
\end{align*}

\end{proof}

\subsection{The degeneration of $\mathbb{S}_{i_1}\ldots \mathbb{S}_{i_l}$ }
In this subsection we prove the following proposition: 

\begin{proposition}\label{DegenerateS}
Let $\mu\in Q^{\vee}$ be arbitrary, and $t^{\mu}=s_{i_1}^{\epsilon_{i_i}}\ldots s_{i_l}^{\epsilon_{i_l}}$, $i_1,\ldots i_l\in \{0,\ldots n\}, \epsilon_i\in \{\pm1\}$ be a reduced decomposition of $t^{\mu}$ in $BW^{aff}$. Then $$\Phi(\mathbb{S}_{i_1}^{\epsilon_{i_i}}\ldots \mathbb{S}_{i_l}^{\epsilon_{i_l}})=C'' (1+\frac{1}{2}\tau(h_{\mu})) \mod Y_{\hbar}\g_{\ge 2}, $$
where $\tau$ is defined in Equation (\ref{tau}), and $C''$ is a scalar function acting as constants $\pm 1$ on every weight space and not depending on  $\lambda, q$ or $\mathbf{z}$.
\end{proposition}

We want to apply this Proposition to $\mu\in Q^{\vee}$ dominant, and $t^{\mu}=s_{i_1}\ldots s_{i_l}$.

\begin{remark}
Gautam and Toledano Laredo conjectured a general formula for the elements of the quantum Weyl group which correspond to lattice elements of the braid group $BW^{aff}$, see \cite{GTL2}  for a discussion of the $\mathfrak{sl}_2$ case.
\end{remark}

\begin{proof}
We will prove this in Lemma \ref{ShortCoroot} for the special case $\mu=\alpha_i^{\vee}$, where $\alpha_i$ is a simple root of the same length as $\theta$, $\alpha_i^{\vee}=w(\theta^{\vee})$ for some $w\in W$. As $\theta$ is always a long root, $\theta^{\vee}$ is always short,  such $\alpha_i^{\vee}$ span $Q^{\vee}$, and the general case follows. 
\end{proof}

We first prove an auxiliary lemma about quantum Weyl group operators. Remember that the map $\Phi$ maps the filtered algebra $U_qL\g$ to the completion of the graded algebra $Y_\hbar \g$, respecting the filtration and grading, and that we are interested in $\Phi(\mathbb{S}_i) \in Y_{\hbar}\g$ up to $(Y_\hbar \g)_{\ge 2}$. For that purpose, we first rewrite $\mathbb{S}_i \in U_qL\g$, disregarding the elements of $\mathcal{J}^2 \subset U_qL\g$.

\begin{lemma}Up to $\mathcal{J}^2$, the quantum Weyl group operator is
\begin{eqnarray*}\mathbb{S}_i=\mathbb{S}(\mathcal{H}_i, \mathcal{E}_i, \mathcal{F}_i)&=& \exp_{q_i^{-1}}(q_i^{-1}\mathcal{E}_iq_{i}^{-\mathcal{H}_i}) \exp_{q_i^{-1}}(-\mathcal{F}_i)\exp_{q_i^{-1}}(q_i\mathcal{E}_iq_{i}^{\mathcal{H}_i}) \, q_i^{\mathcal{H}_i(\mathcal{H}_i+1)/2} \\
&=&\left( 1+\frac{d_i\hbar}{4}(\mathcal{E}_i\mathcal{F}_i+\mathcal{F}_i\mathcal{E}_i-\mathcal{H}_i) \right)\exp(\mathcal{E}_i)\exp(-\mathcal{F}_i)\exp(\mathcal{E}_i) \qquad \textrm{ mod  } \mathcal{J}^2.\end{eqnarray*}
Analogously, 
\begin{eqnarray*}\mathbb{S}_i=\mathbb{S}(\mathcal{H}_i, \mathcal{E}_i, \mathcal{F}_i)&=&
\exp_{q_i^{-1}}(-q_i^{-1}\mathcal{F}_iq_{i}^{\mathcal{H}_i}) \exp_{q_i^{-1}}(\mathcal{E}_i)\exp_{q_i^{-1}}(-q_i\mathcal{F}_iq_{i}^{-\mathcal{H}_i}) \, q_i^{\mathcal{H}_i(\mathcal{H}_i+1)/2} \\
&=&\left( 1+\frac{d_i\hbar}{4}(\mathcal{E}_i\mathcal{F}_i+\mathcal{F}_i\mathcal{E}_i-\mathcal{H}_i) \right)\exp(\mathcal{-F}_i)\exp(\mathcal{E}_i)\exp(\mathcal{-F}_i) \qquad \textrm{ mod  } \mathcal{J}^2.
\end{eqnarray*}\label{QuantumWeylAtqvs1}
\end{lemma}

\begin{proof}
Let us prove the first identity. Calculate each term separately. 
\begin{eqnarray*}
\exp_{q_i^{-1}}(q_i^{-1}\mathcal{E}_iq_{i}^{-\mathcal{H}_i}) &=& \sum_{m\ge 0}\frac{1}{[m]_{q_i}!}q_i^{-m(m-1)/2}(q_i^{-1}\mathcal{E}_iq_{i}^{-\mathcal{H}_i})^m  \\
&=& \sum_{m\ge 0} \frac{1}{m!} \left(1+\frac{d_i\hbar}{2}(-\frac{m(m-1)}{2}-m)\right) (\mathcal{E}_iq_{i}^{-\mathcal{H}_i})^m   \qquad \textrm{ mod  } \mathcal{J}^2 \\
&=& \sum_{m\ge 0} \frac{1}{m!} \left(1+\frac{d_i\hbar}{2}(\frac{m(m-1)}{2}+m(1-\mathcal{H}_i))\right) \mathcal{E}_i ^m   \qquad \textrm{ mod  } \mathcal{J}^2 \\
&=& \left(1+\frac{d_i\hbar}{2}(\frac{1}{2}\mathcal{E}_i^2 +(1-\mathcal{H}_i)\mathcal{E}_i)\right) \exp(\mathcal{E}_i) \qquad \textrm{ mod  } \mathcal{J}^2.
\end{eqnarray*}
Similarly, 
\begin{eqnarray*}
\exp_{q_i^{-1}}(-\mathcal{F}_i) &=& \left(1+\frac{d_i\hbar}{2}(\frac{-1}{2}\mathcal{F}_i^2 ) \right)\exp(\mathcal{-F}_i) \qquad \textrm{ mod  } \mathcal{J}^2 \\
\exp_{q_i^{-1}}(q_i\mathcal{E}_iq_{i}^{\mathcal{H}_i}) & = &\left(1+\frac{d_i\hbar}{2}(\frac{-3}{2}\mathcal{E}_i^2 +(\mathcal{H}_i-1)\mathcal{E}_i)\right) \exp(\mathcal{E}_i)  \textrm{ mod  } \mathcal{J}^2 \\
q_i^{\mathcal{H}_i(\mathcal{H}_i+1)/2}&=& 1+\frac{d_i\hbar}{2}  \frac{\mathcal{H}_i(\mathcal{H}_i+1)}{2} \textrm{ mod  } \mathcal{J}^2.
\end{eqnarray*}
Thus, 
$$\mathbb{S}_i \cdot \left(\exp(\mathcal{E}_i)\exp(-\mathcal{F}_i)\exp(\mathcal{E}_i) \right)^{-1} =$$
$$=\left(1+\frac{d_i\hbar}{2}(\frac{1}{2}\mathcal{E}_i^2 +(1-\mathcal{H}_i)\mathcal{E}_i)\right) \cdot \mathrm{Ad}(\exp (\mathcal{E}_i))\left(1+\frac{d_i\hbar}{2}(\frac{-1}{2}\mathcal{F}_i^2 )\right) \cdot $$
$$\cdot \mathrm{Ad}(\exp (\mathcal{E}_i)\exp (\mathcal{-F}_i))\left(1+\frac{d_i\hbar}{2}(\frac{-3}{2}\mathcal{E}_i^2 +(\mathcal{H}_i-1)\mathcal{E}_i)\right)  \cdot $$
$$\cdot \mathrm{Ad}(\exp (\mathcal{E}_i)\exp (\mathcal{-F}_i)\exp (\mathcal{E}_i))\left(1+\frac{d_i\hbar}{2}  \frac{\mathcal{H}_i(\mathcal{H}_i+1)}{2} \right) \qquad \textrm{ mod  } \mathcal{J}^2. $$

Using that $\mathcal{E}_i, \mathcal{F}_i, \mathcal{H}_i$ satisfy relations for $U\mathfrak{sl}_2$ modulo $\mathcal{J}$, we get that 
\begin{eqnarray*}\mathrm{Ad}(\exp (\mathcal{E}_i))\mathcal{E}_i=\mathcal{E}_i, &  \mathrm{Ad}(\exp (\mathcal{E}_i))\mathcal{H}_i=\mathcal{H}_i-2 \mathcal{E}_i,  &\mathrm{Ad}(\exp (\mathcal{E}_i))\mathcal{F}_i=\mathcal{F}_i+\mathcal{H}_i- \mathcal{E}_i \quad \textrm{ mod  } \mathcal{J}\\
\mathrm{Ad}(\exp (\mathcal{-F}_i))\mathcal{F}_i=\mathcal{F}_i, &  \mathrm{Ad}(\exp (-\mathcal{F}_i))\mathcal{H}_i=\mathcal{H}_i-2 \mathcal{F}_i,  &\mathrm{Ad}(\exp (-\mathcal{F}_i))\mathcal{E}_i=\mathcal{E}_i+\mathcal{H}_i- \mathcal{F}_i \quad \textrm{ mod  } \mathcal{J}.
\end{eqnarray*}

Consequently,
$$\mathrm{Ad}(\exp (\mathcal{E}_i)\exp (\mathcal{-F}_i))(\hbar \mathcal{E}_i)=-\hbar \mathcal{F}_i, \quad  \mathrm{Ad}(\exp (\mathcal{E}_i)\exp (\mathcal{-F}_i))(\hbar\mathcal{H}_i)=\hbar(-\mathcal{H}_i-2\mathcal{F}_i) \quad \textrm{ mod  } \mathcal{J}^2$$
$$\mathrm{Ad}(\exp (\mathcal{E}_i)\exp (\mathcal{-F}_i)\exp (\mathcal{E}_i))(\hbar\mathcal{H}_i)=-\hbar\mathcal{H}_i \quad \textrm{ mod  } \mathcal{J}^2.$$

Finally, 
\begin{eqnarray*}
\mathbb{S}_i&=& \left(1+\frac{d_i\hbar}{2}(\frac{1}{2}\mathcal{E}_i^2 +(1-\mathcal{H}_i)\mathcal{E}_i) \right)\cdot \left(1+\frac{d_i\hbar}{2}(\frac{-1}{2}(\mathcal{F}_i+\mathcal{H}_i- \mathcal{E}_i )^2)\right)\cdot \\
&&\cdot \left(1+\frac{d_i\hbar}{2}(\frac{-3}{2}\mathcal{F}_i^2 +(\mathcal{H}_i+2\mathcal{F}_i+1)\mathcal{F}_i )\right)\cdot \left(1+\frac{d_i\hbar}{2}\frac{\mathcal{H}_i(\mathcal{H}_i-1)}{2} \right)\cdot \\
&&\cdot \left(\exp(\mathcal{E}_i)\exp(-\mathcal{F}_i)\exp(\mathcal{E}_i) \right) \\
&=& \left(1+\frac{d_i\hbar}{4}(\mathcal{E}_i\mathcal{F}_i+\mathcal{F}_i\mathcal{E}_i-\mathcal{H}_i) \right)\cdot \left(\exp(\mathcal{E}_i)\exp(-\mathcal{F}_i)\exp(\mathcal{E}_i) \right) \quad \textrm{ mod  } \mathcal{J}^2.
\end{eqnarray*}

The second identity is proved analogously. 
\end{proof}

\begin{lemma}\label{QuantumWeylSl2}
Proposition \ref{DegenerateS} holds for $\g=\mathfrak{sl}_2$. For $\theta^{\vee}=\theta=\alpha_1$ the positive root, and $t^{\alpha_1^{\vee}}=\mathbb{S}_0\mathbb{S}_1\in BW^{aff}$,
$$\Phi(\mathbb{S}_0\mathbb{S}_1)=C''(1+\frac{1}{2}\tau(h_1))=C''(1-T_{1,1}+\frac{\hbar}{2}h_1^2) \mod Y_{\hbar}\g_{\ge 2}.$$
Here $C''$ is a scalar function acting as a constant $\pm 1$ on every weight space and not depending on $\lambda, q$ or $\mathbf{z}$.
\end{lemma}
\begin{proof}

Write $e,f,h$ for $x_1^+=X_{1,0}^+,x_1^-=X_{1,0}^-,T_{1,0}$ in $\mathfrak{sl}_2\subseteq (Y_\hbar \mathfrak{sl}_2)_{0}$, and $e^{(1)},f^{(1)},h^{(1)}$ for $X_{1,1}^+,X_{1,1}^-,T_{1,1}$ in $(Y_\hbar \mathfrak{sl}_2)_{1}$.  

As in the proof of Lemma \ref{14}, using the definition of $\Phi$ we see that 
\begin{eqnarray*}
\Phi(\mathcal{E}_1)=e, & \Phi(\mathcal{F}_1)=f, & \Phi(\mathcal{H}_1)=-\Phi(\mathcal{H}_0) =h \quad \textrm{ mod } (Y_{\hbar}\mathfrak{sl}_2)_{\ge 2}.
\end{eqnarray*}

The isomorphism between two presentations of the quantum loop algebra depends on the choice of sign associated to the only vertex of the Dynkin diagram $o(1)=\pm 1$. Under this isomorphism, $\mathcal{E}_0=-o(1)q^{-\mathcal{H}_1}F_{1,1}$ and $\mathcal{F}_0=-o(1)E_{1,-1}q^{\mathcal{H}_1}$, so 
\begin{eqnarray*}
\Phi(\mathcal{E}_0)&=-o(1)(1-\frac{\hbar}{2}h)(f+f^{(1)})&=-o(1)(f+f^{(1)}-\frac{\hbar}{2}hf) \quad \textrm{ mod } (Y_{\hbar}\mathfrak{sl}_2)_{\ge 2} \\
\Phi(\mathcal{F}_0)&=-o(1)(e-e^{(1)})(1+\frac{\hbar}{2}h) &=-o(1)(e-e^{(1)}+\frac{\hbar}{2}eh) \quad \textrm{ mod } (Y_{\hbar}\mathfrak{sl}_2)_{\ge 2}.
\end{eqnarray*}
and
$$
\hbar \Phi(\mathcal{E}_0)=-o(1) \hbar f, \quad  \hbar \Phi(\mathcal{F}_0)=-o(1) \hbar e \quad \textrm{ mod } (Y_{\hbar}\mathfrak{sl}_2)_{\ge 2}.
$$

Using Lemma \ref{QuantumWeylAtqvs1} and the fact that $\Phi|_{\mathcal{J}^2}:\mathcal{J}^2\to (Y_{\hbar}\mathfrak{sl}_2)_{\ge 2}$ we get
\begin{eqnarray*}
\Phi(\mathbb{S}_0\mathbb{S}_1) &=&  \left( 1+\frac{\hbar}{4}\left(\Phi\left(\mathcal{E}_0\mathcal{F}_0+\mathcal{F}_0\mathcal{E}_0-\mathcal{H}_0\right)\right) \right)\Phi\left(\exp(\mathcal{E}_0)\exp(-\mathcal{F}_0)\exp(\mathcal{E}_0)\right)\cdot  \\
&& \cdot  \left( 1+\frac{\hbar}{4}\left(\Phi\left(\mathcal{E}_1\mathcal{F}_1+\mathcal{F}_1\mathcal{E}_1-\mathcal{H}_1\right)\right) \right)\Phi\left(\exp(\mathcal{-F}_1)\exp(\mathcal{E}_1)\exp(\mathcal{-F}_1)\right) \quad \textrm{ mod } (Y_{\hbar}\mathfrak{sl}_2)_{\ge 2}
\end{eqnarray*}
Using that
$$\Ad\left(\Phi\left(\exp(\mathcal{E}_0)\exp(-\mathcal{F}_0)\exp(\mathcal{E}_0)\right)\right)\left(\Phi\left(\mathcal{E}_1\mathcal{F}_1+\mathcal{F}_1\mathcal{E}_1-\mathcal{H}_1\right) \right)= ef+fe+h \quad \textrm{ mod } (Y_{\hbar}\mathfrak{sl}_2)_{\ge 1},$$
we get that the above expression is equal to 
\begin{eqnarray*}
\Phi(\mathbb{S}_0\mathbb{S}_1) &=& \left( 1+\frac{\hbar}{2}\left(ef+fe+h\right) \right) \exp(\Phi(\mathcal{E}_0))\exp(\Phi(-\mathcal{F}_0))\exp(\Phi(\mathcal{E}_0))\exp(-f)\exp(e)\exp(-f).
\end{eqnarray*}

Let us first assume that $o(1)=-1$. Then
$$\exp(\Phi(\mathcal{E}_0))\exp(\Phi(-\mathcal{F}_0))\exp(\Phi(\mathcal{E}_0))\exp(-f)\exp(e)\exp(-f)=
$$
$$=\exp(f+f^{(1)}-\frac{\hbar}{2}hf)\exp(-e+e^{(1)}-\frac{\hbar}{2}eh)\exp(f+f^{(1)}-\frac{\hbar}{2}hf)\exp(-f)\exp(e)\exp(-f) \quad \textrm{ mod } (Y_{\hbar}\mathfrak{sl}_2)_{\ge 2}$$

The defining relations of the Yangian imply that $[f,f^{(1)}-\frac{\hbar}{2}hf]=0=[e,e^{(1)}-\frac{\hbar}{2}eh]$. From this one can deduce the following easy identities: 
\begin{eqnarray*}
\exp(f+f^{(1)}-\frac{\hbar}{2}hf)\exp(-f) &=&1+f^{(1)}-\frac{\hbar}{2}hf  \quad \textrm{ mod } (Y_{\hbar}\mathfrak{sl}_2)_{\ge 2}\\
\exp(-e+e^{(1)}-\frac{\hbar}{2}eh)\exp(e) &=&1+e^{(1)}-\frac{\hbar}{2}eh \quad \textrm{ mod } (Y_{\hbar}\mathfrak{sl}_2)_{\ge 2}\\
\mathrm{Ad}\exp(-e+e^{(1)}-\frac{\hbar}{2}eh)x&=&\mathrm{Ad}\exp(-e)x \quad \textrm{ mod } (Y_{\hbar}\mathfrak{sl}_2)_{\ge 2},\quad x\in (Y_{\hbar}\mathfrak{sl}_2)_{\ge 1}\\
\mathrm{Ad}\exp(f+f^{(1)}-\frac{\hbar}{2}hf)x&=&\mathrm{Ad}\exp(f)x \quad \textrm{ mod } (Y_{\hbar}\mathfrak{sl}_2)_{\ge 2},\quad x\in (Y_{\hbar}\mathfrak{sl}_2)_{\ge 1},
\end{eqnarray*}
which we will use along with the following identities (which can also be obtained directly from the defining relations of the Yangian):
\begin{eqnarray*}
\mathrm{Ad}\exp(f)(f^{(1)})&=& f^{(1)}+\hbar f^2 \quad \textrm{ mod } (Y_{\hbar}\mathfrak{sl}_2)_{\ge 2} \\
\mathrm{Ad}\exp(f)(h^{(1)})&=& h^{(1)}+2f^{(1)}+\hbar (hf+fh+3f^2) \quad \textrm{ mod } (Y_{\hbar}\mathfrak{sl}_2)_{\ge 2} \\
\mathrm{Ad}\exp(f)(e^{(1)})&=& e^{(1)}-h^{(1)}-f^{(1)}-\hbar (\frac{hf+fh}{2}+f^2) \quad \textrm{ mod } (Y_{\hbar}\mathfrak{sl}_2)_{\ge 2} \\
\mathrm{Ad}\exp(-e)(e^{(1)})&=& e^{(1)}+\hbar e^2 \quad \textrm{ mod } (Y_{\hbar}\mathfrak{sl}_2)_{\ge 2} \\
\mathrm{Ad}\exp(-e)(h^{(1)})&=& h^{(1)}+2e^{(1)}+\hbar (he+eh+3e^2) \quad \textrm{ mod } (Y_{\hbar}\mathfrak{sl}_2)_{\ge 2} \\
\mathrm{Ad}\exp(-e)(f^{(1)})&=& f^{(1)}-h^{(1)}-e^{(1)}-\hbar (\frac{he+eh}{2}+e^2) \quad \textrm{ mod } (Y_{\hbar}\mathfrak{sl}_2)_{\ge 2}
\end{eqnarray*}
to get that 
$$\exp(\Phi(\mathcal{E}_0))\exp(\Phi(-\mathcal{F}_0))\exp(\Phi(\mathcal{E}_0))\exp(-f)\exp(e)\exp(-f)=$$
$$=\mathrm{Ad}(\exp(f)\exp(-e))(1+f^{(1)}-\frac{\hbar}{2}hf)\cdot \mathrm{Ad}(\exp(f))(1+e^{(1)}-\frac{\hbar}{2}eh) \cdot (1+f^{(1)}-\frac{\hbar}{2}hf)= $$
$$= \mathrm{Ad}(\exp(f))\left( 1+f^{(1)}-h^{(1)}-\frac{\hbar}{2}(hf-h^2+2ef) \right)\cdot (1+f^{(1)}-\frac{\hbar}{2}hf)= $$
$$=1-h^{(1)}+\frac{\hbar}{2}(h^2-2ef).$$
Putting it all together, we get that in this case, 
\begin{eqnarray*}
\Phi(\mathbb{S}_0\mathbb{S}_1) &=& \left( 1+\frac{\hbar}{2}\left(ef+fe+h\right) \right)\left(1-h^{(1)}+\frac{\hbar}{2}(h^2-2ef) \right)  \quad \textrm{ mod } (Y_{\hbar}\mathfrak{sl}_2)_{\ge 2}\\
&=& 1-h^{(1)}+\frac{\hbar}{2}h^2 \quad \textrm{ mod } (Y_{\hbar}\mathfrak{sl}_2)_{\ge 2},
\end{eqnarray*}
which is exactly the claimed $1+\frac{1}{2}\tau(h_1)$. This finishes the proof in case $o(1)=-1$.

Let us now consider the case $o(1)=1$. This is equivalent to working in $o(1)=-1$ case, and calculating
$$\Phi\left( \mathbb{S}(\mathcal{H}_0,-\mathcal{E}_0,-\mathcal{F}_0)\, \mathbb{S}(\mathcal{H}_1,\mathcal{E}_1,\mathcal{F}_1)\right) \quad \textrm{ mod } (Y_{\hbar}\mathfrak{sl}_2)_{\ge 2}.$$

Let $\Phi_1$ be an automorphism of $U_{q}L\mathfrak{sl}_2$ given by $\Phi_1(\mathcal{E}_i)=-\mathcal{E}_i$, $\Phi_1(\mathcal{F}_i)=-\mathcal{F}_i$, $\Phi_1(\mathcal{H}_i)=\mathcal{H}_i$ for $i=0,1$. Let $\varphi_1$ be an automorphism of $Y_{\hbar}\mathfrak{sl}_2$ given by $\varphi_1(X_{1,m}^\pm)=-X_{1,m}^\pm$, $\varphi_1(T_{1,m})=T_{1,m}$. Then $\Phi\circ \Phi_1=\varphi_1 \circ \Phi$. Consequently, 
\begin{eqnarray*}
 \Phi\left( \mathbb{S}(\mathcal{H}_0,-\mathcal{E}_0,-\mathcal{F}_0)\, \mathbb{S}(\mathcal{H}_1,-\mathcal{E}_1,-\mathcal{F}_1)\right)&=& \Phi\circ \Phi_1 \left( \mathbb{S}(\mathcal{H}_0,-\mathcal{E}_0,-\mathcal{F}_0)\mathbb{S}(\mathcal{H}_1,-\mathcal{E}_1,-\mathcal{F}_1)\right) \\
 &=& \varphi_1\left( 1-h^{(1)}+\frac{\hbar}{2}h^2  \right)   \quad \textrm{ mod } (Y_{\hbar}\mathfrak{sl}_2)_{\ge 2}\\
 &=&  1-h^{(1)}+\frac{\hbar}{2}h^2 \quad \textrm{ mod } (Y_{\hbar}\mathfrak{sl}_2)_{\ge 2}.
\end{eqnarray*}

Finally,  
$$\Phi\left(\mathbb{S}(\mathcal{H}_1,-\mathcal{E}_1,-\mathcal{F}_1)^{-1} \mathbb{S}(\mathcal{H}_1,\mathcal{E}_1,\mathcal{F}_1) \right) =$$ 
$$=(\exp(-e) \exp(f) \exp(-e))^{-1} (1-\frac{\hbar}{4}(ef+fe-h))(1+\frac{\hbar}{4}(ef+fe-h))\exp(e) \exp(-f) \exp(e) $$
$$=\left(\exp(e) \exp(-f) \exp(e)\right)^2 \quad \textrm{ mod } (Y_{\hbar}\mathfrak{sl}_2)_{\ge 2}.$$
This is the square of the quantum Weyl group operator $\exp(e) \exp(-f) \exp(e)$ for the copy of $\mathfrak{sl}_2\subseteq (Y_\hbar \mathfrak{sl}_2)_{0}$. We can restrict any representation of $Y_\hbar \mathfrak{sl}_2$ to this copy of $\mathfrak{sl}_2$, and the square of the quantum Weyl group element acts as a scalar $(-1)^h$ on every weight space of this restricted representations. Calling this scalar function $C''$, we get that indeed, in case of $o(1)=1$, 
$$\Phi\left( \mathbb{S}(\mathcal{H}_0,-\mathcal{E}_0,-\mathcal{F}_0)\mathbb{S}(\mathcal{H}_1,\mathcal{E}_1,\mathcal{F}_1)\right)=C'' (1-h^{(1)}+\frac{\hbar}{2}h^2)\quad \textrm{ mod } (Y_{\hbar}\mathfrak{sl}_2)_{\ge 2}.$$
\end{proof}

\begin{lemma} \label{ShortCoroot}
Let $\alpha_i$ be a simple root of $\g$ the same length as the highest root $\theta$, $\alpha_i^{\vee}$ the corresponding coroot, $t^{\alpha_i^{\vee}}$ the corresponding element of $BW^{aff}$, and $T_{\alpha_i^{\vee}}$ its representation in the quantum Weyl group. Then 
$$\Phi(T_{\alpha_i^{\vee}})=C''\left(1+\frac{1}{2}\tau(h_i) \right) \quad \textrm{ mod } (Y_{\hbar}\mathfrak{g})_{\ge 2},$$
where $C''$ is a scalar function acting as a constant $\pm 1$ on every weight space and not depending on $\lambda, q$ or $\mathbf{z}$.
\end{lemma}
\begin{proof}
As  $\alpha_i$ is the same length as $\theta$, and the action of $W$ on $\theta$ is transitive, there exists $w\in W$ such that $w(\alpha_i)=\theta$, $w(\alpha_i^{\vee})=\theta^{\vee}$. From Section \ref{LatticeElts} it follows that then $\mathcal{E}_0=T_wT_{\omega_i^{\vee}}T_i^{-1}(\mathcal{E}_i), \, \mathcal{F}_0=T_wT_{\omega_i^{\vee}}T_i^{-1}(\mathcal{F}_i)$.

Let $w=s_{i_1}\ldots s_{i_l}$ be a reduced decomposition in $W$. Then the following holds in $BW^{aff}$:
\begin{eqnarray*}
s_{\theta}&=& s_{i_1}\ldots s_{i_l} s_i s_{i_l}\ldots s_{i_1}\\
t^{\theta^{\vee}}&=&s_0s_{\theta}\\
t^{\alpha_i^{\vee}}&=&s_{i_l}^{-1}\ldots s_{i_1}^{-1}t^{\theta^{\vee}}s_{i_1}^{-1}\ldots s_{i_l}^{-1}.
\end{eqnarray*}

Thus, in the representation by quantum Weyl group elements,
\begin{eqnarray*}
T_{\theta^{\vee}}&=& \mathbb{S}_0\mathbb{S}_{i_1}\ldots \mathbb{S}_{i_l}\mathbb{S}_i \mathbb{S}_{i_l}\ldots \mathbb{S}_{i_1}\\
T_{\alpha_i^{\vee}}&=& \mathbb{S}_{i_l}^{-1}\ldots \mathbb{S}_{i_1}^{-1}T_{\theta^{\vee}}\mathbb{S}_{i_1}^{-1}\ldots \mathbb{S}_{i_l}^{-1}\\
&=& \mathbb{S}_{i_l}^{-1}\ldots \mathbb{S}_{i_1}^{-1}\mathbb{S}_0 \mathbb{S}_{i_1}\ldots \mathbb{S}_{i_l}\mathbb{S}_i\\
&=& T_w^{-1}(\mathbb{S}_0)\mathbb{S}_i\\
&=& T_w^{-1}(\mathbb{S}(\mathcal{H}_0,\mathcal{E}_0,\mathcal{F}_0)) \mathbb{S}(\mathcal{H}_i,\mathcal{E}_i,\mathcal{F}_i)\\
&=& T_w^{-1}(\mathbb{S}(T_wT_{\omega_i^{\vee}}T_i^{-1}(\mathcal{H}_i),T_wT_{\omega_i^{\vee}}T_i^{-1}(\mathcal{E}_i),T_wT_{\omega_i^{\vee}}T_i^{-1}(\mathcal{F}_i))) \mathbb{S}(\mathcal{H}_i,\mathcal{E}_i,\mathcal{F}_i)\\
&=& \mathbb{S}(T_{\omega_i^{\vee}}T_i^{-1}(\mathcal{H}_i),T_{\omega_i^{\vee}}T_i^{-1}(\mathcal{E}_i),T_{\omega_i^{\vee}}T_i^{-1}(\mathcal{F}_i)) \mathbb{S}(\mathcal{H}_i,\mathcal{E}_i,\mathcal{F}_i)
\end{eqnarray*}
Proposition 3.8. in \cite{B} shows that $\mathcal{H}_0\mapsto T_{\omega_i^{\vee}}T_i^{-1}(\mathcal{H}_i)$, $\mathcal{E}_0\mapsto T_{\omega_i^{\vee}}T_i^{-1}(\mathcal{E}_i)$, $\mathcal{F}_0\mapsto T_{\omega_i^{\vee}}T_i^{-1}(\mathcal{F}_i)$,$\mathcal{H}_1\mapsto \mathcal{H}_i$, $\mathcal{E}_1\mapsto \mathcal{E}_i$, $\mathcal{F}_1\mapsto \mathcal{F}_i$ gives an inclusion of $U_{q_i}L\mathfrak{sl}_2$ into $U_qL\g$. Thus, we can use Lemma \ref{QuantumWeylSl2} to conclude that 
$$T_{\alpha_i^{\vee}}=C''\left(1+\frac{1}{2}\tau(h_i) \right) \quad \textrm{ mod } (Y_{\hbar}\mathfrak{g})_{\ge 2}.$$

\end{proof}

\subsection{The scalar $CC''$ }

In Corollary \ref{WriteProduct} we have shown  that $\mathcal{A}^{\mu}(\mathbf{z},\lambda)=C\cdot \mathbb{S}_{i_1}\ldots \mathbb{S}_{i_l} \cdot \mathbb{B}_{s_{i_1}\ldots s_{i_l}}(\lambda).$ In Lemma \ref{16} we calculated $\Phi(\mathbb{B}_{s_{i_1}\ldots s_{i_l}}(\lambda))$ up to $Y_{\hbar}\g_{\ge 2}$. In Proposition \ref{DegenerateS} we proved that up to $Y_{\hbar}\g_{\ge 2}$, $\Phi(\mathbb{S}_{i_1}\ldots \mathbb{S}_{i_l})=C'' (1+\frac{1}{2}\tau(h_{\mu}))$. To calculate $\Phi(\mathcal{A}^{\mu}(\mathbf{z},\lambda))$ up to $Y_{\hbar}\g_{\ge 2}$, we need to calculate the value of $C C''$, which is a scalar function not depending on  $\lambda, q$ or $\mathbf{z}$.

\begin{lemma}\label{C=1}
The scalar $CC''$ is identically equal to $1$.
\end{lemma}
\begin{proof}
As it does not depend on $\lambda, q$ or $\mathbf{z}$, it is enough to prove this at $q=1$ (meaning up to $\hbar$), and in the limit $\lambda \to \infty$. 

Putting together Corollary \ref{WriteProduct}, Lemma \ref{16} and Proposition \ref{DegenerateS}, we get that
\begin{equation}
\Phi(\mathcal{A}^{\mu}(\mathbf{z},\lambda))=CC'' \cdot (1+\frac{1}{2}\tau(h_{\mu}))\cdot \Phi(\mathbb{B}_{s_{i_1}\ldots s_{i_l}}(\lambda)).\label{constants}\end{equation}

At $q=1$ and in the limit $\lambda \to \infty$, the right hand side of (\ref{constants}) equals
\begin{equation}
CC'' \cdot (1-h_{\mu}u)\cdot 1.\label{rhs}\end{equation}


Let us now calculate the left hand side, by calculating $\mathcal{A}^{\mu}(\mathbf{z},\lambda)$, for $q=1$, $\lambda\to \infty$. By definition, $\mathcal{A}_i (\mathbf{z},\lambda)=\Pi_i^q \mathcal{A}_{w_i} (\lambda)$, $\lim_{\lambda\to \infty} \mathcal{A}_{w} (\lambda)=A_{w}^{+}$, , and  $\Pi_i^q|_{q=1}=z^{-\omega_i^{\vee}}(A_{w_i}^{+})^{-1}$. Thus,
$$\lim_{\lambda\to \infty} \mathcal{A}_i (\mathbf{z},\lambda)|_{q=1}=\Pi_i^q A_{w_i}^{+}=z^{-\omega_i^{\vee}}$$ and
\begin{align}\lim_{\lambda\to \infty}\mathcal{A}^{\mu}(\mathbf{z},\lambda)|_{q=1} =\prod_i \lim_{\lambda\to \infty} \mathcal{A}_i (\mathbf{z},\lambda)^{k_i}|_{q=1}=\prod_i (z^{-\omega_i^{\vee}})^{k_i}=z^{-\mu}. \end{align}
At $q=1$, $\Phi$ is just substitution $z=e^u$, so
\begin{equation}
\Phi(\lim_{\lambda\to \infty}\mathcal{A}^{\mu}(\mathbf{z},\lambda)|_{q=1}) = \Phi(z^{-\mu})=1-\mu u \quad \mod u^2. \label{lhs} \end{equation}

Comparing \ref{lhs} and \ref{rhs}, and remembering that $h_{\mu}$ corresponds to $\mu$ under the identification  $\h\cong \h^*$, we get that 
$CC'' \equiv 1$.

\end{proof}


\subsection{Proof of Theorem \ref{main}}\label{main proof}
\begin{proof}
First, assume $\mu\in Q^{\vee}$ is dominant. Then 
\begin{align*}
\Phi(\mathcal{A}^\mu (\mathbf{z}, \frac{\lambda}{\hbar}))&=C\cdot  \Phi(\mathbb{S}_{i_1}\ldots \mathbb{S}_{i_l}) \cdot \Phi(\mathbb{B}_{s_{i_1}\ldots s_{i_l}}(\frac{\lambda}{\hbar})) \quad \textrm{(using Corollary \ref{WriteProduct})} \\
&=CC''\cdot (1+\frac{1}{2}\tau(h_{\mu}))\cdot (1+\hbar \sum_{\alpha \in \mathbf{R}_+} \frac{\left<\alpha,\mu\right>}{e^{-\left< \lambda, \alpha \right>}-1} x_{\alpha}^-x_{\alpha}^+)  \mod Y_{\hbar}\g_{\ge 2} \quad \textrm{(using Prop \ref{16} and \ref{DegenerateS}}) \\
&=1+\frac{1}{2}\tau(h_{\mu})+\hbar \sum_{\alpha \in \mathbf{R}_+} \frac{\left<\alpha,\mu\right>}{e^{-\left< \lambda, \alpha \right>}-1} x_{\alpha}^-x_{\alpha}^+  \mod Y_{\hbar}\g_{\ge 2} \quad \textrm{(using Lemma \ref{C=1}}).\end{align*}

Next, for $-\mu\in Q^{\vee}$ antidominant, 
\begin{align*}
\Phi(\mathcal{A}^{-\mu}(\mathbf{z}, \frac{\lambda}{\hbar}))&=\Phi(\mathcal{A}^\mu (\mathbf{z}, \frac{\lambda}{\hbar}-\mu))^{-1} \\
&=1-\frac{1}{2}\tau(h_{\mu})-\hbar \sum_{\alpha \in \mathbf{R}_+} \frac{\left<\alpha,\mu\right>}{e^{-\left< \lambda, \alpha \right>}-1} x_{\alpha}^-x_{\alpha}^+  \mod Y_{\hbar}\g_{\ge 2}. \end{align*}

Finally, for arbitrary $\mu\in Q^{\vee}$, write it as $\mu=\mu_+-\mu_-$, with $\mu_+,\mu_-$ both dominant. Then 
\begin{align}
\Phi(\mathcal{A}^\mu (\mathbf{z}, \frac{\lambda}{\hbar}))&=\Phi(\mathcal{A}^{\mu_+} (\mathbf{z}, \frac{\lambda}{\hbar}-\mu_-))\Phi(\mathcal{A}^{\mu_-}(\mathbf{z}, \frac{\lambda}{\hbar})) \nonumber \\
&=\left( 1+\frac{1}{2}\tau(h_{\mu_+})+\hbar \sum_{\alpha \in \mathbf{R}_+} \frac{\left<\alpha,\mu_+\right>}{e^{-\left< \lambda, \alpha \right>}-1} x_{\alpha}^-x_{\alpha}^+\right) \cdot  \nonumber \\
& \qquad \qquad  \cdot \left(1-\frac{1}{2}\tau(h_{\mu_-})-\hbar \sum_{\alpha \in \mathbf{R}_+} \frac{\left<\alpha,\mu_-\right>}{e^{-\left< \lambda, \alpha \right>}-1} x_{\alpha}^-x_{\alpha}^+ \right) \mod Y_{\hbar}\g_{\ge 2}\nonumber\\
&=1+\frac{1}{2}\tau(h_{\mu})+\hbar \sum_{\alpha \in \mathbf{R}_+} \frac{\left<\alpha,\mu\right>}{e^{-\left< \lambda, \alpha \right>}-1} x_{\alpha}^-x_{\alpha}^+  \mod Y_{\hbar}\g_{\ge 2}.
\end{align}

This finishes the proof of the main Theorem \ref{main}.

\end{proof}




\end{document}